\documentclass[11pt]{extarticle}
\usepackage[table]{xcolor}
\usepackage{diagbox} 

\usepackage{amsmath,amsthm,amssymb,amsbsy,amsopn,mathtools,comment,bbm}
\usepackage{graphicx,color,float,subcaption,caption}
\usepackage[bookmarks=true,bookmarksnumbered=true,bookmarkstype=toc]{hyperref}
\usepackage{mathptmx,bm,bbm}

\usepackage[ruled,vlined]{algorithm2e}

\usepackage{lineno}

\usepackage{tikz}
\usetikzlibrary{decorations.pathreplacing,shapes,arrows,automata,positioning}

\usepackage{booktabs}
\usepackage{multirow} % For Option 2

\usepackage[titletoc]{appendix}

\graphicspath{{../fig/}}

\def\argmin{\mathop{\rm argmin}}

\newtheorem{Def}{Definition}[section]
\newtheorem{theorem}[Def]{Theorem}
\newtheorem{lemma}[Def]{Lemma}
\newtheorem{proposition}[Def]{Proposition}

\newtheorem{remark}[Def]{Remark}
\newtheorem{Exam}[Def]{Example}
\newtheorem{assumption}[Def]{Assumption}

\theoremstyle{definition}

\makeatletter
\@addtoreset{equation}{section}
\makeatother

\title{Stopping rules for Monte Carlo methods of martingale difference type}%\footnotetext[0]{This version: \today. This work was initiated while both authors were based in School of Mathematics and Statistics at the University of Sydney.}}
\author{\sc Jiezhong Wu\footnote{Email: jiezhongwu2021@u.northwestern.edu.
Affiliation: Department of Industrial Engineering and Management Sciences, Northwestern University, Evanston, IL 60208, USA.} \, and Reiichiro Kawai\footnote{Corresponding Author. Email: raykawai@g.ecc.u-tokyo.ac.jp.
Affiliation: Graduate School of Arts and Sciences / Mathematics and Informatics Center, The University of Tokyo, Japan.
This work was partially supported by JSPS Grants-in-Aid for Scientific Research 20K22301 and 21K03347.}}
\date{}%\today}
\begin{document}

%\linenumbers

\maketitle 

\begin{abstract}
\noindent
We establish a practical and easy-to-implement sequential stopping rule for the martingale central limit theorem, focusing on Monte Carlo methods for estimating the mean of a non-iid sequence of martingale difference type. Starting with an impractical scheme based on the standard martingale central limit theorem, we progressively address its limitations from implementation perspectives in the non-asymptotic regime. Along the way, we compare the proposed schemes with their counterparts in the asymptotic regime. The developed framework has potential applications in various domains. Numerical results are provided to demonstrate the effectiveness of the developed stopping rules in terms of reliability and complexity.

\vspace{0.3em}

\vspace{0.3em}
\end{abstract}

%\textcolor{red}{
%The reason that I was computing $\mathbb{E}_{k-1}[\xi_k^2]$ in \ref{section ARCH(1)} was to classify which case the example belongs to (see page 5 before \ref{section ARCH(1)}). Now given a modified martingale central limit theorem, we don't have to do the classification, and therefore we can drop computations of $\mathbb{E}_{k-1}[\xi_k^2]$ in all $4$ examples. Right?
%}

\section{Introduction}\label{introduction}

In Monte Carlo methods in general terms, it is a common practice to fix the total number of samples before running the simulation, while, without sufficient knowledge of the underlying distribution, it could be risky to terminate the procedure at such a priori fixed timing outright for the reason that the estimator that one is dealing with may converge extremely slowly.
Moreover, the length of the simulation needs to be decided in reference to various budgetary constraints, as computing resources available are finite in practice without exception.
Those issues can be addressed by equipping the simulation procedure with a sequential (or, equivalently, adaptive) stopping rule, so that the total number of samples can be decided along the way depending on the previous computation,
%in which sample size is determined by the outcome of the simulation and that indicates when to stop generating randomness 
%in order to achieve a desired degree of accuracy appropriately, 
rather than stopping the simulation too soon or continuing the simulation endlessly.
%Nonetheless, it is not trivial when to stop the iteration in general.
%and central limit theorem is not quite of use for this purpose. 

In the framework of the standard iid Monte Carlo simulation based on the central limit theorem, there has been a great deal of interest in such stopping criteria in the literature for a very long time.
In the asymptotic regime, sequential stopping rules were first developed based on the first and second moments and investigated in terms of consistency and convergence rates (for instance, \cite{YH1965, PW1992, NS1996, RD1976}, to name just a few).
%and confidence intervals for the mean and median in symmetric location parameter population \cite{RD1976}.
%while in general they are unable to secure an accuracy in non-asymptotic regimes. 
%Sequential stopping rules for confidence interval for the mean and median in symmetric location parameter population was developed in \cite{RD1976} and it is also favorable in asymptotic regimes. 
Nevertheless, in the non-asymptotic regime where the cost of computing needs to be taken into account, 
%the stopping criterion developed in \cite{FLYA2012} based on a priori knowledge of a upper bound for the kurtosis of the underlying distribution and the algorithm is proven to fail \cite{RL1956} when no moments bounds are given in prior to sampling. 
such asymptotics-based stopping rules may not function as theoretically expected, without extra information and restrictions, for instance, higher moments and boundedness of the underlying distribution (see, for example, \cite{BHST, LWP2013, FLYA2012} and references therein), due to an inaccurate estimate of the error probability at an early stage of the simulation procedure.
%Even when the individual samples are not independent, series correlation method and subgrouping approach are poor unless one has thorough knowledge of the simulated system \cite{M1968} whereas the method of spectral analysis can be applied successfully \cite{GP1967, O1994}. 

Departing from the standard iid Monte Carlo simulation and its underlying standard central limit theorem, non-iid sequences have also attracted attention in the context of sequential stopping rules (for instance, \cite{LWP2013, M1968, doi:10.1287/mnsc.28.5.550, stoppingrules}) and, more recently, machine learning.
Among a variety of such non-iid sequences, the focus of the present work is set on a sequence of martingale difference type, which comes with the martingale central limit theorem.
To be exact, we aim to establish a rigorous yet easy-to-implement sequential stopping rule for the martingale central limit theorem with a view towards Monte Carlo methods in estimating the mean of a sequence of random variables (just like iid Monte Carlo methods), whose law, however, keeps changing in reference to the past information within a martingale framework (unlike iid Monte Carlo methods).

The relevant situation arises, for instance, from adaptive variance reduction techniques employed concurrently with the computation of the empirical mean so that the underlying randomness can be used in both procedures, such as importance sampling \cite{RK_1, RK18}, control variates %\cite{adaptive2007, KAWAI2020123608, saaccelmcBAVR}, 
and stratified sampling \cite{
%Etore2011, 
Kawai:2010:AOA:1734222.1734225, 
ISCVSS}, to mention just a few.
Due to their remarkable effectiveness in reducing the estimator variance without requiring much additional computing effort for running the parallel procedure, those adaptive variance reduction methods have found various fields of application, such as structural engineering \cite{ASDIS, ReliabilitySS} and financial risk management \cite{Glasserman2004}. 
The proposed framework on stopping rules for martingale difference sequences also has significant potential in machine learning optimization, particularly stochastic gradient descent methods where the stochastic gradient noises naturally form a martingale difference sequence.
While conventional stopping methods often rely on validation set performance \cite{prechelt2012}, our approach provides a principled alternative based on the statistical properties of the optimization trajectory itself.
This is particularly of relevance given the recent variance reduction techniques for stochastic gradient descent methods \cite{johnson2013accelerating}. 
We also highlight non-asymptotic bounds for the martingale central limit theorem established for averaged stochastic gradient descent \cite{anastasiou2019, shao2022}, offering a theoretical foundation complementary to the practical stopping criteria discussed in the present study.

In constructing valid stopping rules of this martingale type, the primary challenge lies in the impracticality of applying the martingale central limit theorem in its standard form.
Unlike the standard central limit theorem for iid Monte Carlo methods, where the limiting variance is merely unknown, the theoretical variance underpinning the martingale central limit theorem (such as in \cite{martingalebook2022}) is not only unknown but also cumulative averaging.
In addition, the theoretical variance here is unconditional, that is, poorly aligned with the conditional expectations inherent to martingales.
Moreover, the martingale central limit theorem generally hinges on intricate and often unverifiable technical conditions, which further complicates its practical use.
Given those difficulties, we explore an implementable stopping rule by progressively replacing the theoretical variance with more implementable alternatives, while carefully imposing technical conditions along the way for the corresponding stopping rules to remain valid.
In particular, in its final form, the developed stopping rule operates upon a fully implementable empirical variance.

The rest of this paper is set out as follows.
To begin with, in Section \ref{section problem setup}, we prepare the general notation and set up standing assumptions along with some relevant examples in our context. % (Example \ref{examples of problem settings}). 
In Section \ref{section problem formulation}, we describe the problem formulation and review the martingale central limit theorem in its standard form.
%modified version for practical use, which is followed by a description of the second-moments based algorithm. 
%It can be shown that the second moments based algorithm is asymptotically consistent. 
We construct in Section \ref{section main section} sequential stopping rules based on the martingale central limit theorem with its theoretical batch variance replaced with more implementable ones step by step, first with the conditional variance (Section \ref{subsection Stopping rule based on conditional batch variance}) and then the empirical variance (Section \ref{subsection Stopping rule based on empirical batch variance}).
%In Section \ref{subsection Stopping rule based on uniform bound of Berry-Esseen type}, we construct a sequential stopping rule with the aid of the implementable Berry-Esseen bound for the martingale central limit theorem \cite{sydney} with a view towards the practical application in the non-asymptotic regime.
We provide numerical results in Section \ref{section numerical illustations} to demonstrate the effectiveness of the proposed stopping rule on a time series (Section \ref{section ARCH(1)}) and an adaptive variance reduction method (Section \ref{section adaptive control variates}).
%and \ref{section adaptive importance sampling}).
We conclude the present study by highlighting future research directions in Section \ref{section concluding remarks}.
To maintain the flow of the paper, we collect the proofs in Section \ref{section proofs}.
Finally, in the supplementary material, we provide verifications of the standing assumptions and conditions for the numerical examples examined in Section \ref{section numerical illustations} and discuss relevant stopping rules to put the proposed framework into perspective.

\section{Preliminaries}\label{section problem setup} %{22222}

We begin with some general notation used throughout this study.
We use the notation $\mathbb{N}:=\{1,2,\cdots\}$ and $\mathbb{N}_0:=\mathbb{N}\cap \{0\}$. 
We denote by $|\cdot|$ and $\| \cdot \|$, respectively, the magnitude and the Euclidean norm.
%\textcolor{red}{For a real valued function $f$, we write $f^{p}(x):=|f(x)|^p$ for the sake of simplicity.} 
For any linear topological space $V$, we denote by $\mathcal{B}(V)$ the Borel $\sigma$-algebra defined on $V$. 
%Let $\mathbb{P}$ and $\mathbb{E}$ be the probability measure and corresponding expectation we work with.
We denote by $\mathbbm{1}(D)$ the indicator function associated with the event $D$.
We denote by $\nabla_{\mathbf{x}} \text { and } \mathrm{Hess}_{\mathbf{x}}$ the gradient and the Hessian matrix with respect to the multivariate variable $\mathbf{x}$.
We denote by $\mathbb{I}_d$, $\mathbbm{1}_d$ and $0_d$, respectively, the identity matrix of order $d$, the vector in $\mathbb{R}^d$ with all unit-valued components, and the zero vector in $\mathbb{R}^d$.
We reserve $\phi, \Phi$, and $\Phi^{-1}$, respectively, for the standard normal density function, the standard normal cumulative distribution function, and its inverse.
We write $f(n)\sim g(n)$ for $f(n)/g(n)\to 1$ as $n$ tends to a specified value (including $\pm \infty$), $f(n)\lesssim g(n)$ if $f(n)\le g(n)$ for sufficiently large/small $n$ as specified, and $f(n)\simeq g(n)$ if there exists a positive constant $c$ such that $f(n)/c\le g(n)\le cf(n)$, again for sufficiently large/small $n$ as specified.

%Notice that this setting is different from that of \cite{BHST}, in the sense that we do not start over every time and at each $t\in\mathbb{N}$, all the unknowns $\{X_k\}_{k\in M_{t-1}}$ at $t-1$ become known.
We reserve the notation $\{X_k\}_{k\in\mathbb{N}}$ for a sequence of random variables of interest.
Based upon the sequence $\{X_k\}_{k\in\mathbb{N}}$, we work on a filtered probability space, denoted by $(\Omega, \mathcal{F},(\mathcal{F}_k)_{k\in\mathbb{N}_0}, \mathbb{P})$ on which a sequence of random variables $\{X_k\}_{k\in\mathbb{N}}$ is defined.
In particular, the filtration starts from the $\sigma$-field $\mathcal{F}_0$, which is a collection of $\mathbb{P}$-null sets, that is $\mathcal{F}_{0}:=\sigma(\mathcal{N})$ with $\mathcal{N}:=\{N \subseteq \Omega: \exists A \in \vee_{k \in \mathbb{N}} \mathcal{F}_k, \mathbb{P}(A)=0, N \subseteq A\}$. 
Note that the filtration $(\mathcal{F}_k)_{k\in\mathbb{N}_0}$ is not necessarily generated by the sequence $\{X_k\}_{k\in\mathbb{N}}$.
Moreover, for the sake of simplicity and convenience, we write  $\mathbb{P}_k(\cdot):=\mathbb{P}(\cdot|\mathcal{F}_k)$, $\mathbb{E}_k[\cdot]:=\mathbb{E}[\cdot|\mathcal{F}_k]$, and $\mathrm{Var}_k(\cdot):=\mathrm{Var}(\cdot|\mathcal{F}_k)$ for all $k\in \mathbb{N}_0$.
%Note that $\mathbb{E}_0[\cdot]$ and $\mathbb{P}_0(\cdot)$ are conditional on the $\sigma \text {-field } \mathcal{F}_{0}$, yet unconditional in effect. 
In addition, we let $\stackrel{\mathbb{P}_0}{\to}$ denote the convergence in probability under $\mathbb{P}_0$ and indicate by $G_k\stackrel{\mathbb{P}_0}{\to}+\infty$ (for a sequence of random variables $\{G_k\}_{k\in\mathbb{N}}$) that $\lim_{k\to +\infty}\mathbb{P}_0(|G_k|>c)=1$ for all $c>0$.

Now, we impose the base assumption on the sequence $\{X_k\}_{k\in\mathbb{N}}$ of random variables, which remains true throughout the present study.
%Let $\{X_k\}_{k\in\mathbb{N}}$ be a sequence that satisfies \eqref{base assumption} and at each time $t$, we use the $t$th batch, $\{X_k\}_{k\in M(t)}$.

\begin{assumption}\label{standing assumption 1}
For every $k\in \mathbb{N}$, the random variable $X_k$ is $\mathcal{F}_k$-measurable with
%Let $\{X_k\}_{k\in\mathbb{N}}$ be an array of random variables and we define a filtration generated by $\{X_k\}_{k\in\mathbb{N}}$ by $(\mathcal{F}_k)_{k\in\mathbb{N}}$. That is, for each $m\in\mathbb{N}$, $\mathcal{F}_{m}=\sigma(\{X_k\}_{k\in\{1, \ldots, m\}})$ is the $\sigma$-algebra generated by the first $m$ iid random variables $\{X_k\}_{k\in\{1,\cdots,m\}}$. The sequence $\{X_k\}_{k\in\mathbb{N}}$ is required to satisfy that for each $k\in\mathbb{N}$,
\begin{equation}\label{base assumption}
\mathbb{E}_{k-1}\left[X_k\right]=\mu,
\end{equation}
for some $\mu\in \mathbb{R}$.
%and 
%\begin{equation}\label{moments assumption}
%\sup\limits_{k\in\mathbb{N}}\mathbb{E}_{k-1}\left[|X_k|^3\right]<+\infty, \quad \mathbb{P}_0\text{-}a.s.
%\end{equation}
\end{assumption}

Clearly, the standing assumption \eqref{base assumption} is much weaker than the usual iid assumption and not to indicate that the sequence $\{X_k\}_{k\in\mathbb{N}}$ is a martingale (for which it would have to satisfy $\mathbb{E}_{k-1}[X_k]=X_{k-1}$ for all $k\in \mathbb{N}$ with suitable $X_0\in \mathcal{F}_0$).
Instead, the assumption \eqref{base assumption} can be interpreted to form a wider class than martingale difference sequences since the sequence $\{X_k\}_{k\in\mathbb{N}}$ is a martingale difference sequence if, moreover, $\mu=0$.
It is, however, inappropriate to assume $\mu=0$ upfront or consider $Y_k:=X_k-\mu$ instead, since the constant $\mu$ is deemed unknown and is exactly what one looks to estimate in the present context.
We also note that a similar assumption has been made in \cite[Assumption 1.1]{LWP2013} in the context of relative-precision stopping rules, but along with a boundedness condition on the underlying randomness (that is, $X_n\in [0,1]$, $a.s.$).

To further illustrate the scope of the standing assumption \eqref{base assumption}, we describe a few relevant problem settings below.

\begin{Exam}{\rm \label{examples of problem settings}
(a) If $\{X_k\}_{k\in \mathbb{N}}$ is a sequence of iid random variables with finite mean $\mu$ and if $(\mathcal{F}_k)_{k\in\mathbb{N}_0}$ is its natural filtration, then it satisfies Assumption \ref{standing assumption 1}.
In this sense, our developments in what follows can be thought of as generalizations of the existing stopping rules for iid Monte Carlo methods.
At this stage, however, we wish to emphasize that such sequences of iid random variables are outside the scope of the present study for the reason that such a way too simple situation can be treated based on much lighter conditions (for instance, \cite{BHST, YH1965, M1968, FLYA2012, NS1996, RD1976}) than we impose in what follows for the general problem setting \eqref{base assumption}.

%such that the condition \eqref{moments assumption} holds with a natural filtration $\mathcal{F}_k:=\sigma(X_1,\cdots,X_k)$ automatically satisfies the conditions of martingale difference sequence although we are not interested to apply the proposed stopping rules for iid case for obvious reasons. 

% \noindent \textcolor{blue}{(b) Let $\{\xi\}_{k\in \mathbb{N}}$ be a sequence of iid random variables taking values $\{-1,+1\}$ equally likely with its natural filtration $(\mathcal{F}_k)_{k\in\mathbb{N}_0}$ and let $V_k$ be an $\mathcal{F}_{k-1}$-measurable random variable taking values $\{0,1,2,3\}$ for all $k\in \mathbb{N}$.
% If $X_k=\mu+V_k \xi_k$, then its running sum $\sum_{k=1}^n X_k$ represents a biased random walk with increments that are neither independent nor identically distributed.
% Assumption \ref{standing assumption 1} is satisfied.
% From a practical point of view, we assume that the experimenter can only generate the increments collectively as $X_k$ (that is, not through the individual components $\mu$, $V_k$ or $\xi_k$) and but also retrospectively, meaning that the $(n+1)$st increment can be freely regenerated based on the previous increments $\{X_k\}_{k\in \{1,\cdots,n\}}$.}
%We examine various technical conditions on this example in Section \ref{section random walk}.}
 
\noindent (b) Consider an ARCH(1) sequence $\{X_k\}_{k\in\mathbb{N}}$, defined %recursively 
by $X_k:=(\beta+\alpha X_{k-1}^2)^{1/2} V_k$ for some $\alpha\in (0,1)$ and $\beta\in (0,+\infty)$, where $\{V_k\}_{k\in\mathbb{N}}$ is a sequence of iid random variables with zero mean and unit variance.
If $(\mathcal{F}_k)_{k\in\mathbb{N}_0}$ is the natural filtration generated by the sequence $\{V_k\}_{k\in\mathbb{N}}$, then it holds true that $\mathbb{E}_{k-1}[X_k]=0$ for all $k\in\mathbb{N}$, which satisfies Assumption \ref{standing assumption 1} with $\mu=0$.
Later in Section \ref{section ARCH(1)}, we examine our stopping rules on this problem setting.
 
\noindent (c) A variety of adaptive variance reduction methods (such as \cite{RK_1, RK18}) %, KAWAI2020123608} 
are also within our scope, as described in the introductory section.
Those frameworks can be formulated in common as follows.
Let $\{U_k\}_{k\in\mathbb{N}}$ be a sequence of iid uniform random vectors on $(0,1)^d$ and let $(\mathcal{F}_k)_{k\in\mathbb{N}_0}$ be its natural filtration.
For all $k\in\mathbb{N}$, we set $X_k=R(U_k;{\bm \theta}_{k-1})$, where $R(\cdot;{\bm \theta})$ is an unbiased estimator for $\mu$, in the sense of $\mathbb{E}_0[R(U;{\bm \theta})]=\mu$ for $U\sim U(0,1)^d$, irrespective of the parameter ${\bm \theta}$, corresponding to importance sampling and/or control variates.
If the sequence $\{{\bm \theta}_k\}_{k\in\mathbb{N}_0}$ of the parameters is constructed in an adaptive manner to the filtration $(\mathcal{F}_k)_{k\in\mathbb{N}}$, then it holds true that $\mathbb{E}_{k-1}[X_k]=\mu$, due to ${\bm \theta}_{k-1}\in \mathcal{F}_{k-1}$ and $U_k\perp \mathcal{F}_{k-1}$ for all $k\in\mathbb{N}$.
The stopping rules that we develop are examined on adaptive control variates 
%and importance sampling, respectively, 
in Section \ref{section adaptive control variates}.
%and \ref{section adaptive importance sampling}.
\qed}\end{Exam}

Hereafter, we reserve the notation $t$ for the order of batches taking its value in $\mathbb{N}_0$.
Fix a sequence of non-negative deterministic integers $\{m(t)\}_{t\in\mathbb{N}_0}$ satisfying $m(0)=0$ and $m(t)>m(t-1)$ for all $t\in \mathbb{N}$, that is, $\lim_{t\to +\infty}m(t)=+\infty$.
We denote by $M(t):=\{m(t-1)+1,\cdots,m(t)\}$ and its cardinality $|M(t)|:=m(t)-m(t-1)$, respectively, the set of indices in the $t$th batch and the $t$th batch size, for all $t\in \mathbb{N}$.
We formalize the assumption that the size of batches $|M(t)|$ grows to infinity, as below.

\begin{assumption}\label{standing assumption M}
We henceforth assume $|M(t)|\to +\infty$, as $t\to +\infty$.
\end{assumption}

Now, in order to describe where stopping rules may play a role, we first define the empirical batch mean by
\begin{equation}\label{empirical batch mean}
\mu(t):=\frac{1}{|M(t)|}\sum_{k\in M(t)} X_k,\quad t\in\mathbb{N}.
\end{equation}
Henceforth, we indicate the batch index $t$ as a variable of function (that is, $(t)$), whether deterministic (such as $M(t)$ and $m(t)$) or probabilistic (such as $\mu(t)$), so as to distinguish from the subscript index $k$ (like, $X_k$).

Below, we state the standard results of (i) unbiasedness and (ii) strong consistency of the empirical batch mean \eqref{empirical batch mean}, with a sketchy proof given in Section \ref{section proofs}. 
Note that the denominator in the condition of (ii) is not $k^2$ but $(k-m(t-1))^2$, due to the batch arrangement $M(t)=\{m(t-1)+1,\cdots,m(t)\}$.

\begin{theorem}\label{theorem law of large numbers}
(i) It holds $\mathbb{P}_0$-$a.s.$ that $\mathbb{E}_{m(t-1)}[\mu(t)]=\mathbb{E}_0[\mu(t)]=\mu$ for all $t\in \mathbb{N}$.

\noindent (ii)
If $\limsup_{t\to +\infty}\sum_{k\in M(t)}(k-m(t-1))^{-2}{\rm Var}_0(X_k)<+\infty$, then it holds $\mathbb{P}_0$-$a.s.$ that $\mu(t)\to \mu$, as $t\to+\infty$.
\end{theorem}

%Let us make essential notes on the relevance of the proposed experiment \eqref{empirical batch mean}.
%Given that we are working on a single sequence $\{X_k\}_{k\in\mathbb{N}}$ (rather than an array $\{X_{n,k}\}$ in the usual context of the martingale central limit theorem), 
The experiment \eqref{empirical batch mean} can be significantly different for different batches, because it is not based on iid realizations but on a single martingale difference sequence (not even an array), depending largely on its past history.
Even so, Theorem \ref{theorem law of large numbers} provides a firm theoretical base that ensures our problem setting makes perfect sense in estimating the unknown mean $\mu$, regardless of the batch in which one is computing the empirical batch mean \eqref{empirical batch mean}. 
%\textcolor{purple}{The difference between the usual context of MCLT and this single run of simulation}

\section{Problem formulation}\label{section problem formulation}

In light of the unbiasedness and strong consistency of the empirical batch mean \eqref{empirical batch mean} (Theorem \ref{theorem law of large numbers}), we aim to construct a suitable rule for stopping the simulation at the $t$th batch,
%ideally defined as follows: 
%\begin{equation}\label{fundamental definition of tau}
%\tau\leftarrow 
%\min_{t\in\mathbb{N}}\left\{\mathbb{P}_0\left(|\mu(t)-\mu|>\epsilon\right)\le \delta\right\},
%\end{equation}
%with precision $\epsilon\in (0,+\infty)$ and error probability $\delta\in (0,1)$ fixed, and then output a suitable empirical batch mean.
%In plain words, we wish to stop the experiment at a batch
in which the error probability $\mathbb{P}_0(|\mu(t)-\mu|>\epsilon)$ is satisfactorily small (than a preset acceptable error probability $\delta\in (0,1)$ with a preset precision level $\epsilon\in (0,+\infty)$) for the first time.
Intuitively, we need to correspond smaller values of the precision $\epsilon$ and error probability $\delta$ to larger values of the stopping batch $\tau$.
From a practical point of view, however, this criterion is not quite useful per se, as the error probability $\mathbb{P}_0(|\mu(t)-\mu|>\epsilon)$ is generally not assessable unless the law of the distance $|\mu(t)-\mu|$ is known. 
It is rather absurd to assume such knowledge, as one is conducting experiments because the constant $\mu$ is unknown in the first place.

On top of the non-assessability of the error probability, it is also a crucial drawback that the unbiasedness (Theorem \ref{theorem law of large numbers} (i)) is generally no longer valid for a general $(\mathcal{F}_{m(t)})_{t\in \mathbb{N}_0}$-stopping time $\tau$, if the empirical batch mean is given as an output at the stopping batch (that is, $\mathbb{E}_0[\mu(\tau)]\ne \mu$).
As an unrealistic yet illustrative example, suppose the support of the random variable $X_k$ under $\mathbb{P}_0$ is $\mathbb{R}$ for every $k\in \mathbb{N}$ and the stopping batch $\tau$ is defined as the first batch $t$ in which all elements $\{X_k\}_{k\in M(t)}$ are strictly positive.
Then, for every $k\in M(\tau)$ in the stopping batch, the support of $X_k$ is $(0,+\infty)$.
%, whereas that of the regenerated element $X_k^{\star}$ remains to be $\mathbb{R}$, just as originally.
If $\mu=0$ (despite it is unknown in reality), then the original empirical batch mean at the stopping batch, which is now strictly positive by definition, is evidently biased for the constant $\mu$ of interest, that is, $\mathbb{E}_0[\mu(\tau)]>0=\mu$, $\mathbb{P}_0$-$a.s.$ 

To address this issue, we output the regenerated empirical batch mean at the stopping batch, instead.
In order to formalize its formulation, we prepare the notation.
We let $\tau$ denote an $(\mathcal{F}_{m(t)})_{t\in\mathbb{N}_0}$-stopping time taking non-negative integer values, that is, the event $\{\tau\in \{0,1,\cdots,t\}\}$ is $\mathcal{F}_{m(t)}$-measurable for all $t\in \mathbb{N}_0$.
We denote by $\mathcal{F}_{m(\tau)}$ the $\sigma$-field at the $(\mathcal{F}_{m(t)})_{t\in\mathbb{N}_0}$-stopping time $\tau$, that is, $\mathcal{F}_{m(\tau)}:=\{B\in \mathcal{F}:\,B\cap \{\tau\in \{0,1,\cdots,k\}\}\in \mathcal{F}_{m(k)} \text{ for all }k\in \mathbb{N}_0\}.$
Clearly, the tail $\{X_k\}_{k\in\{m(\tau)+1,\cdots\}}$ of the sequence is independent of the $\sigma$-field $\mathcal{F}_{m(\tau)}$.
The truncated sequence $\{X_k\}_{k\in\{1,\cdots,m(\tau)\}}$ is $\mathcal{F}_{m(\tau)}$-measurable.
Now, if the stopping batch $\tau$ is determined, then regenerate a sequence of random variables $\{X^{\star}_k\}_{k\in M(\tau)}$ for the stopping $\tau$th batch as a continuation of the past $\{X_k\}_{k\in \{1,\cdots,m(\tau-1)\}}$, that is, we are given the sequence 
\[
 (X_1,\cdots,X_{m(1)},X_{m(1)+1},\cdots,X_{m(\tau-1)},X^{\star}_{m(\tau-1)+1},\cdots,X^{\star}_{m(\tau)}),
\]
with which we compute and output the empirical batch mean $\mu_{\star}(\tau)$ of the regenerated batch $(X^{\star}_{m(\tau-1)+1},\cdots,X^{\star}_{m(\tau)})$, defined by \begin{equation}\label{def of mu*}
\mu_{\star}(\tau):=\frac{1}{|M(\tau)|}\sum_{k\in M(\tau)}X^{\star}_k.
\end{equation}
For illustrative purposes, we provide Figure \ref{figure what to output} to illustrate the procedure for constructing the regenerated empirical batch mean \eqref{def of mu*}.

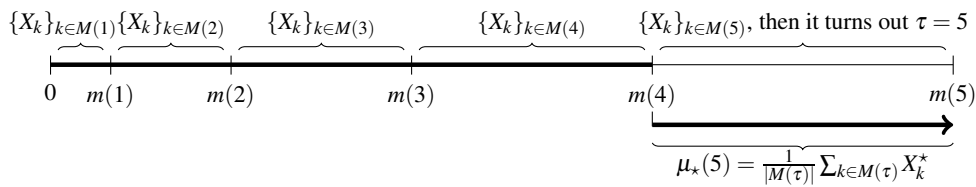
\begin{figure}[ht]
	\centering\scriptsize
	\begin{tikzpicture}[scale=0.8]
    %xscale=0.8, %1.1, 
    %yscale=0.8]%1.1]
	\tikzset{
		position label/.style={
			below = 3pt,
			text height = 1.5ex,
			text depth = 1ex
		},
		brace/.style={
			decoration={brace}, %, mirror},
			decorate
		}
	}
	
	\draw[line width=0.6mm, -] (0,1.5) node[below, yshift = -0.5em] {\footnotesize $0$} -- (1,1.5) node [below, yshift = -0.5em] {\footnotesize $m(1)$} -- (3,1.5) node [below, yshift = -0.5em] {\footnotesize $m(2)$} -- (6,1.5) node [below, yshift = -0.5em] {\footnotesize $m(3)$} -- (10,1.5) node [below, yshift = -0.5em] {\footnotesize $m(4)$};
	\draw[-] (10,1.5) -- (15,1.5) node [below, yshift = -0.5em] {\footnotesize $m(5)$};
	\foreach \x in {0,1,3,6,10,15} \draw (\x cm,47pt) -- (\x cm,38pt);
	
	\node [position label] (c0) at (0,2.2){};
	\node [position label] (c1) at (1,2.2){};
	\node [position label] (c2) at (3,2.2){};
	\node [position label] (c3) at (6,2.2){};
	\node [position label] (c4) at (10,2.2){};
	\node [position label] (c5) at (15,2.2){};

	\draw [brace] (c0) -- (c1) node [position label, above=0.2em, pos=0.1] {\footnotesize $\{X_k\}_{k\in M(1)}$};
	\draw [brace] (c1) -- (c2) node [position label, above=0.2em, pos=0.5] {\footnotesize $\{X_k\}_{k\in M(2)}$};
	\draw [brace] (c2) -- (c3) node [position label, above=0.2em, pos=0.5] {\footnotesize $\{X_k\}_{k\in M(3)}$};
	\draw [brace] (c3) -- (c4) node [position label, above=0.2em, pos=0.5] {\footnotesize $\{X_k\}_{k\in M(4)}$};
	\draw [brace] (c4) -- (c5) node [position label, above=0.2em, pos=0.5] {\footnotesize $\{X_k\}_{k\in M(5)}$, then it turns out $\tau=5$};
	
	\foreach \x in {10} \draw (\x cm,20.5pt) -- (\x cm,13.5pt);
	
	\draw[line width=0.6mm,->] (10,0.50) -- (15,0.50) ;
    \draw [brace] (15,0.20) -- (10,0.20) node [position label, below=0.3em, pos=0.5] {\footnotesize $\mu_{\star}(5)=\frac{1}{|M(\tau)|}\sum_{k\in M(\tau)}X_k^{\star}$};

\end{tikzpicture}
\caption{Typical arrangement of batches ($m(t)$ and $M(t)$), the batched samples $\{X_k\}_{k\in M(t)}$, and the regenerated batch mean $\mu_{\star}(\tau)$. 
Recall that $m(0)=0$, $m(t)>m(t-1)$, $M(t)=\{m(t-1)+1,\cdots,m(t)\}$ for all $t\in \mathbb{N}$, and $|M(t)|\to +\infty$ (Assumption \ref{standing assumption M}).
Now, suppose it turns out $\tau=5$ based upon the past and latest information $\mathcal{F}_{m(5)}$.
The output of our stopping rules is not the empirical batch mean $\mu(5)(=|M(5)|^{-1}\sum_{k\in M(5)}X_k)$ along the top line, but the regenerated empirical batch mean $\mu_{\star}(5)$ over the $5$th batch along the second line carrying on the past information $\mathcal{F}_{m(4)}$ (indicated by the thick arrow).
The regenerated sequence $\{X_k^{\star}\}_{k\in M(\tau)}$ is conditionally independent of the original counterpart $\{X_k\}_{k\in M(\tau)}$ on the stopping-time $\sigma$-field $\mathcal{F}_{m(\tau)}$.}
%\textcolor{red}{need rethinking. not sure if one can simply carry on the past $\mathcal{F}_{m(4)}$.}}
\label{figure what to output}
\end{figure}

The motivation for introducing the regenerated empirical batch mean $\mu_{\star}(\tau)$ is nothing but its unbiasedness while the one $\mu(\tau)$ based on the original sequence $\{X_k\}_{k\in \mathbb{N}}$ can be biased.
To be exact, although the stopping batch $\tau$ is determined based on the original sequence $\{X_k\}_{k\in \mathbb{N}}$ (the top line of Figure \ref{figure what to output}), the output needs to be the regenerated empirical batch mean $\mu_{\star}(\tau)$ at the stopping batch $\tau$ to ensure unbiasedness, as formalized below.
%the original empirical batch mean $\mu(\tau)$ may no longer be unbiased.

\begin{theorem}\label{theorem unbiasedness}
It holds $\mathbb{P}_0$-$a.s.$ that  $\mathbb{E}_{m(\tau)}[\mu_{\star}(\tau)]%\textcolor{red}{=\mathbb{E}_0[\mu_{\star}(\tau)]}
=\mu$.
%If $\limsup_{t\to +\infty}\sum_{k\in M(t)}(k-m(t-1))^{-2}{\rm Var}_0(X_k)<+\infty$, then it holds $\mathbb{P}_0$-$a.s.$ that $\mu(t)\to \mu$ as $t\to+\infty$.
\end{theorem}

To distinguish the two random sources $\{X_k\}_{k\in M(\tau)}$ and $\{X^{\star}_k\}_{k\in M(\tau)}$ correctly, we hereafter denote by $\mathbb{P}^{\star}_{k-1}$ the probability measure on the $\sigma$-field $\mathcal{F}^{\star}_{k-1}$ that makes the branched sequence $(X_1,\cdots,X_{m(\tau-1)},X^{\star}_{m(\tau-1)+1},\cdots,X^{\star}_{k-1})$ measurable for all $k\in M(\tau)$ and leaves the original stopping batch $(X_{m(\tau-1)+1},\cdots,X_{m(\tau)})$ out of its measurability.
Accordingly, we write $\mathbb{E}_{k-1}^{\star}$ and ${\rm Var}_{k-1}^{\star}$, respectively, for the expectation and variance taken under the probability measure $\mathbb{P}^{\star}_{k-1}$.
We note that $\mathcal{F}_{k-1}^{\star}=\mathcal{F}_{k-1}$,  $\mathbb{P}_{k-1}^{\star}=\mathbb{P}_{k-1}$, $\mathbb{E}_{k-1}^{\star}=\mathbb{E}_{k-1}$ and ${\rm Var}_{k-1}^{\star}={\rm Var}_{k-1}$ if $k=m(\tau-1)+1$.

Before closing this section, we provide a more detailed illustration of how to construct the regenerated empirical batch mean \eqref{def of mu*}, as depicted in Figure \ref{figure what to output}, based upon the three problem settings already outlined in Example \ref{examples of problem settings}.

\begin{Exam}{\rm \label{examples of regenerating}
Here, we demonstrate the procedure for constructing the regenerated empirical batch mean \eqref{def of mu*} along the concrete problem settings presented in Example \ref{examples of problem settings}.
As illustrated in Figure \ref{figure what to output}, suppose it turns out $\tau=5$ along the sequence $\{X_k\}_{k\in\mathbb{N}}$.
%Again, we write $\tau$ for $\tau_{\epsilon,\delta}$ to ease the notation.

\noindent (a) If $\{X_k\}_{k\in \mathbb{N}}$ is a sequence of iid random variables, then generate the sequence of iid random variables $\{X^{\star}_k\}_{k\in M(\tau)}$, each under the same law as $\mathcal{L}(X_1)$ and independently of the original sequence $\{X_k\}_{k\in\mathbb{N}}$.
This is exactly what is done in the existing stopping rules for the standard iid Monte Carlo simulation (for instance, \cite{BHST}).

%such that the condition \eqref{moments assumption} holds with a natural filtration $\mathcal{F}_k:=\sigma(X_1,\cdots,X_k)$ automatically satisfies the conditions of martingale difference sequence although we are not interested to apply the proposed stopping rules for iid case for obvious reasons. 

% \noindent
%  \textcolor{blue}{(b) If $X_k=\mu+V_k \xi_k$ for the non-standard random walk, then resample the sequence $\{X^{\star}_k\}_{k\in M(\tau)}$ upon all the previous batches to ensure that its ingredient $V_k^{\star}$ is measurable with respect to the $\sigma$-field generated by $\{\xi_l\}_{l\in \{1,\cdots,m(\tau-1)\}}$ and $\{\xi^{\star}\}_{l\in \{m(\tau-1)+1,\cdots,k\}}$ for all $k\in M(\tau)$.}
 %is constructed, as follows:
% \[
%  X^{\star}_k=\begin{dcases}
%  (\beta+\alpha (X_{k-1})^2)^{1/2} V^{\star}_k,& \text{if }k=m(\tau-1)+1,\\
%  (\beta+\alpha (X^{\star}_{k-1})^2)^{1/2} V^{\star}_k,& \text{if }k\in \{m(\tau-1)+2,\cdots,m(\tau)\},
%  \end{dcases}
% \]
% where $\{V^{\star}_k\}_{k\in\mathbb{N}}$ is an independent copy of the sequence $\{V_k\}_{k\in\mathbb{N}}$ of iid random variables with zero mean and unit variance.
% Then, $\{X^{\star}_k\}_{k\in M(\tau)}$ is a regenerated sequence over the $\tau$th batch.}

\noindent (b) If $\{X_k\}_{k\in\mathbb{N}}$ is an ARCH(1) sequence with $X_k=(\beta+\alpha (X_{k-1})^2)^{1/2} V_k$, then construct the sequence $\{X^{\star}_k\}_{k\in M(\tau)}$, as follows:
\[
 X^{\star}_k=\begin{dcases}
 (\beta+\alpha (X_{k-1})^2)^{1/2} V^{\star}_k,& \text{if }k=m(\tau-1)+1,\\
 (\beta+\alpha (X^{\star}_{k-1})^2)^{1/2} V^{\star}_k,& \text{if }k\in \{m(\tau-1)+2,\cdots,m(\tau)\},
 \end{dcases}
\]
where $\{V^{\star}_k\}_{k\in\mathbb{N}}$ is an independent copy of the sequence $\{V_k\}_{k\in\mathbb{N}}$ of iid random variables with zero mean and unit variance.
Then, $\{X^{\star}_k\}_{k\in M(\tau)}$ is a regenerated sequence over the $\tau$th batch. 

\noindent (c) In the context of adaptive control variates and importance sampling, generate the sequence $\{U^{\star}_k\}_{k\in M(\tau)}$ of iid uniform random variables on $(0,1)$, independent of the original sequence $\{U_k\}_{k\in M(\tau)}$.
Accordingly, in a similar manner to the construction of the augmented $\sigma$-field $\mathcal{F}^{\star}_k$, branch the natural filtration $(\mathcal{F}_k)_{k\in \{0,1,\cdots,m(\tau-1)\}}$ on the basis of the regenerated sequence $\{U^{\star}_k\}_{k\in M(\tau)}$, as well as construct $\{{\bm \theta}^{\star}_k\}_{k\in M(\tau)}$ in an adaptive manner to this branched filtration.
Then, construct
\[
 X^{\star}_k=\begin{dcases}
 R(U^{\star}_k;{\bm \theta}_{k-1}),&\text{if }k=m(\tau-1)+1,\\
 R(U^{\star}_k;{\bm \theta}^{\star}_{k-1}),&\text{if }k\in \{m(\tau-1)+2,\cdots,m(\tau)\},
 \end{dcases}
\]
so that $\{X^{\star}_k\}_{k\in M(\tau)}$ is a regenerated sequence over the $\tau$th batch.
\qed}\end{Exam}

\section{The stopping rules}\label{section main section}

The main idea of the present work is to replace the assessment of the error probability $\mathbb{P}_0(|\mu(t)-\mu|>\epsilon)$ with more implementable quantities %step by step (Sections \ref{subsection Stopping rule based on conditional batch variance} and \ref{subsection Stopping rule based on empirical batch variance}) 
%and then a more accurate approximation (Section \ref{subsection Stopping rule based on uniform bound of Berry-Esseen type}) 
for constructing a practical stopping rule.
Along this line, such a quantity that we first work with is the (scaled) theoretical batch variance and standard deviation, respectively, by
\begin{align}\nonumber
v^2_0(t)&:=|M(t)|{\rm Var}_0(\mu(t))=\frac{1}{|M(t)|}\sum_{k\in M(t)} {\rm Var}_0\left(X_k\right)\\
&=\frac{1}{|M(t)|}\sum_{k\in M(t)} \mathbb{E}_0\left[X_k^2 \right]-\mu^2,\label{theoretical batch variance}
\end{align}
and $v_0(t):=\left(v^2_0(t)\right)^{1/2}$ for $t\in \mathbb{N}$, where the first equality holds true by Theorem \ref{theorem law of large numbers} (i) and the tower property of conditional expectation.
Note that the quantities \eqref{theoretical batch variance} are still not implementable in general, whereas they play an important role in our development.
We first impose the following assumptions on its asymptotic behavior.

\begin{assumption}\label{standing assumption 2}
In addition to Assumptions \ref{standing assumption 1} and \ref{standing assumption M}, we impose  $\inf_{t\in\mathbb{N}}v^2_0(t)>0$, $\sup_{t\in \mathbb{N}}v^2_0(t)<+\infty$ and $v^2_0(t+1)/|M(t+1)|\sim v^2_0(t)/|M(t)|$ as $t\to +\infty$, throughout.
\end{assumption}

Plainly speaking, the three assumptions here describe that over every single batch, its sample is neither completely degenerate nor exploding in variance, as well as the batch variance stays in a common order in the long run. 
It is worth noting that Theorem \ref{theorem law of large numbers} (ii) holds true under the second assumption ($\sup_{t\in \mathbb{N}}v^2_0(t)<+\infty$) here, since the condition 
($\limsup_{t\to +\infty}\sum_{k\in M(t)}(k-m(t-1))^{-2}{\rm Var}_0(X_k)<+\infty$) in Theorem \ref{theorem law of large numbers} (ii) is weaker.
%than $\limsup_{t\to +\infty}v^2_0(t)<+\infty$.
For instance, with ${\rm Var}_0(X_k)\sim (k-m(t-1))^q$ for any $q\in (0,1)$, we have $v^2_0(t)\to +\infty$ while the condition $\limsup_{t\to +\infty}\sum_{k\in M(t)}(k-m(t-1))^{-2}{\rm Var}_0(X_k)<+\infty$ holds true.
Hence, the convergence $\lim_{t\to +\infty}\mu(t)= \mu$ remains true $\mathbb{P}_0$-$a.s.$, in what follows.
The third assumption ($v^2_0(t+1)/|M(t+1)|\sim v^2_0(t)/|M(t)|$) does not look intuitive but is rather technical for asymptotic validation of the proposed stopping rules.
%(Theorems \ref{proposition asymptotic validity 1}, \ref{proposition asymptotic validity 2} and \ref{proposition asymptotic validity 3}).
Here, we present a standard form of the martingale central limit theorem \cite{martingalebook2022} in our notation.

\begin{theorem}\label{theorem very base MCLT}
If $|M(t)|^{-1} \sum_{k\in M(t)} \mathbb{E}_{k-1}[|X_k-\mu|^2 \mathbbm{1}(|X_k-\mu|>\epsilon \sqrt{|M(t)|})]\stackrel{\mathbb{P}_0}{\to}0$ as $t\to +\infty$ for all $\epsilon>0$, then it holds under $\mathbb{P}_0$ that
\begin{equation}\label{very first MCLT}
\sqrt{|M(t)|}\frac{\mu(t)-\mu}{v_0(t)} \stackrel{\mathcal{L}}{\to} \mathcal{N}(0,1),
\end{equation}
as $t\to +\infty$.
\end{theorem}

\subsection{The stopping rule based on theoretical batch variances}

In light of the weak convergence \eqref{very first MCLT}, we start with the following stopping rule and its approximation with ``no'' view towards implementation:
\begin{equation}\label{def of tau0}
\tau^0_{\epsilon,\delta}:=\min\left\{t\in\mathbb{N}:\,\mathbb{P}_0(|\mu(t)-\mu|> \epsilon)\le \delta\right\},
\end{equation}
with precision $\epsilon\in (0,+\infty)$ and error probability $\delta\in (0,1)$ specified in subscript for asymptotic justification, unlike in Section \ref{section problem formulation}. 
As is obvious from its expression, the stopping batch $\tau_{\epsilon,\delta}^0$ is $\mathcal{F}_0$-measurable, that is, deterministic in effect, as well as is far from implementable because the law of $|\mu(t)-\mu|$ is unknown.
Still, at this initiation phase, we justify its ``asymptotic'' relevance in a similar manner to \cite{PW1992}, as if the precision $\epsilon$ and the error probability $\delta$ could be adjusted after the sequence $\{X_k\}_{k\in \mathbb{N}}$ of infinite length had been generated.

\begin{proposition}\label{proposition asymptotic validity of tau0}
(i) It holds that  $\tau^0_{\epsilon,\delta}\to +\infty$ as $\epsilon\to 0+$ or $\delta\to 0+$.

\noindent For (ii) and (iii), let the conditions of Theorem \ref{theorem very base MCLT} hold and fix $\delta\in (0,1)$.

\noindent (ii) It holds that $2(1-\Phi(\epsilon\sqrt{|M(\tau^0_{\epsilon,\delta})|}/v_0(\tau^0_{\epsilon,\delta})))\to \delta$, as $\epsilon \to 0+$.

\noindent (iii) If $v^2_0(t)/|M(t)|\sim c t^{-2q}$ as $t\to +\infty$ for some positive $c$ and $q$, then it holds that $\epsilon^{1/q}\tau^0_{\epsilon,\delta}\to (\sqrt{c}\Phi^{-1}(1-\delta/2))^{1/q}$ and $\epsilon^{-1}(\mu(\tau^0_{\epsilon,\delta})-\mu)\stackrel{\mathcal{L}}{\to} \mathcal{N}(0,1/(\Phi^{-1}(1-\delta/2))^2)$, as $\epsilon\to 0+$.
\end{proposition}

In short, it is confirmed in Part (i) that the simulation needs to go on forever if the criterion \eqref{def of tau0} is conducted rigidly ($\epsilon\to 0+$ or $\delta\to 0+$), just as it should be in principle.
Moreover, let us emphasize the divergence $\tau^0_{\epsilon,\delta}\to +\infty$ here is purely deterministic, with nothing to do with probabilistic convergence modes, such as almost sure or in probability.
Part (ii) justifies the replacement of the stopping rule $\tau_{\epsilon,\delta}^0$ by the following approximation in reference to the weak convergence \eqref{very first MCLT}, for sufficiently small precision $\epsilon$:
\begin{align}
\tau^0_{\epsilon,\delta}&=\min\left\{t\in\mathbb{N}:\,\mathbb{P}_0(|\mu(t)-\mu|> \epsilon)\le \delta\right\}\nonumber\\
&\sim \min\left\{t\in\mathbb{N}:\,2\left(1-\Phi\left(\frac{\epsilon \sqrt{|M(t)|}}{v_0(t)}\right)\right)\le \delta\right\},\label{tau0 and approximation}
\end{align}
where the fraction inside the standard normal distribution function is implementable for all $t\in \mathbb{N}$ due to $\inf_{t\in \mathbb{N}}v^2_0(t)>0$ (Assumption \ref{standing assumption 2}). 
Indeed, the effectiveness of this approximation is the key building block for all the following developments in the upcoming sections.
%\ref{subsection Stopping rule based on conditional batch variance}, \ref{subsection Stopping rule based on empirical batch variance} and \ref{subsection Stopping rule based on uniform bound of Berry-Esseen type}.
As for Part (iii), the first result indicates that the stopping batch increases in the order of $\epsilon^{-1/q}$ as the precision $\epsilon$ tends to vanish, depending on the exponent $q$ in the assumption $v^2_0(t)/|M(t)|\sim ct^{-2q}$.
In the context of typical non-batched experiments with empirical mean $n^{-1}\sum_{k=1}^n X_k$ and convergent empirical variance, this assumption often holds with $q=1/2$, resulting in the asymptotic behavior of the stopping rule $\tau^0_{\epsilon,\delta}$ like $\epsilon^{-2}$.
Next, the second result of Part (iii) describes the second-order structure of the output $\mu(\tau_{\epsilon,\delta}^0)$ with the limiting variance $1/(\Phi^{-1}(1-\delta/2))^2$.
That is, the experiment is expected to be more precise (that is, a smaller limiting variance) when the value of the error probability $\delta$ is very small, just as desired.

In turn, for implementation purposes, that is, in the non-asymptotic regime with both $\epsilon$ and $\delta$ fixed, the stopping rule and its approximation \eqref{tau0 and approximation} need to be reformulated in a way to approximate the failure probability $\mathbb{P}_0(|\mu(t)-\mu|>\epsilon)$ by the conditional probability $\mathbb{P}_{m(\tau^0_{\epsilon,\delta})}(|\mu_{\star}(\tau^0_{\epsilon,\delta})-\mu|>\epsilon)$, in the sense of 
\begin{align}
&\mathbb{P}_{m(\tau^0_{\epsilon,\delta})}\left(|\mu_{\star}(\tau^0_{\epsilon,\delta})-\mu|>\epsilon\right)\nonumber\\
&\quad =\mathbb{P}_{m(\tau^0_{\epsilon,\delta})}\left(\sqrt{|M(\tau^0_{\epsilon,\delta})|}
 \frac{|\mu_{\star}(\tau^0_{\epsilon,\delta})-\mu|}{v_0(\tau^0_{\epsilon,\delta})}>\sqrt{|M(\tau^0_{\epsilon,\delta})|}\frac{\epsilon}{v_0(\tau^0_{\epsilon,\delta})}\right)\nonumber\\
 &\quad \approx 2\left(1-\Phi(\epsilon \sqrt{|M(\tau^0_{\epsilon,\delta})|}/v_0(\tau^0_{\epsilon,\delta}))\right),\label{approximation by v_0(t)}
\end{align}
provided that the value of the stopping batch $\tau_{\epsilon,\delta}^0$ is luckily far away.
We stress that the last approximation in \eqref{approximation by v_0(t)} is valid because the regenerated batch $\{X_k^{\star}\}_{k\in M(\tau)}$ is conditionally independent of the stopping batch $\tau_{\epsilon,\delta}^0$ on the stopping-time $\sigma$-field $\mathcal{F}_{m(\tau_{\epsilon,\delta}^0)}$. 

\subsection{The stopping rule based on implementable batch variances}

As is now obvious from its expression, the crucial drawback of the approximate stopping rule \eqref{approximation by v_0(t)} is the unavailability of the theoretical batch standard deviation $\{v_0(t)\}_{t\in \mathbb{N}}$.
In what follows, we aim to replace the theoretical batch standard deviation in the approximate criterion \eqref{tau0 and approximation} with a more implementable sequence $\{v(t)\}_{t\in \mathbb{N}}$ 
%s (Sections \ref{subsection Stopping rule based on conditional batch variance} and \ref{subsection Stopping rule based on empirical batch variance}) 
and make the resulting criterion more precise for a better application of the following martingale central limit theorem 
%\eqref{very first MCLT}.
%We are ready to present the following martingale central limit theorem 
\cite[Chapter VIII.1]{P2012}, which is a direct consequence of Theorem \ref{theorem very base MCLT} owning to the continuous mapping theorem.

\begin{theorem}\label{theorem base MCLT}
Let $\{v(t)\}_{t\in \mathbb{N}}$ be an $(\mathcal{F}_{m(t)})_{t\in \mathbb{N}}$-adapted sequence of random variables satisfying $\inf_{t\in\mathbb{N}}v^2(t)>0$, $\sup_{t\in\mathbb{N}}v^2(t)<+\infty$ and $v^2(t)/v^2_0(t)\stackrel{\mathbb{P}_0}{\to}1$, as $t\to+\infty$, and let the conditions of Theorem \ref{theorem very base MCLT} hold.
Then, it holds under $\mathbb{P}_0$ that 
\begin{equation}\label{first MCLT}
\sqrt{|M(t)|}\frac{\mu(t)-\mu}{v(t)} \stackrel{\mathcal{L}}{\to}\mathcal{N}(0,1),
\end{equation}
as $t\to +\infty$.
\end{theorem}

In the literature, a variety of sufficient conditions have been provided in different forms for the weak convergence \eqref{first MCLT} to hold true, which we do not pursue as those are not directly relevant in our context.
For detail, we refer the reader to, for instance, \cite[Theorem 2.5]{ish}.
Let us stress that the condition $v^2(t)/v^2_0(t)\stackrel{\mathbb{P}_0}{\to} 1$ implies that the sequence $\{v(t)\}_{t\in\mathbb{N}}$ does not vanish in the limit, due to Assumption \ref{standing assumption 2}.
From a practical point of view, this is a reasonable property, since the sample $\{X_k\}_{k\in\mathbb{N}}$ itself (not the sample average) satisfying $v^2(t)\to 0$ would otherwise necessarily converge to the constant $\mu$ of interest, for which a stopping rule would no longer be of particular use. 

With the sequence $\{v(t)\}_{t\in\mathbb{N}}$ (in place to the impractical one $\{v_0(t)\}_{t\in\mathbb{N}}$), the approximate stopping rule \eqref{tau0 and approximation} can be advanced a step further to the following stopping rule:
\begin{equation}\label{stopping based on general v(t)}
 \tau_{\epsilon,\delta}:=\min\left\{t\in \mathbb{N}:\,2\left(1-\Phi\left(\frac{\epsilon \sqrt{|M(t)|}}{v(t)+a(t)}\right)\right)\le \delta\right\},
\end{equation}
where $a(t)$ denotes a strictly positive function that decreases monotonically to zero as $t\to +\infty$, equipped for the purpose of asymptotic justification for the moment.
Later in Section \ref{section numerical illustations}, we touch on the effect of the term $a(t)$ on the performance of the stopping rule \eqref{stopping based on general v(t)}.

Let us emphasize that the sequence $\{v(t)\}_{t\in \mathbb{N}}$ in \eqref{stopping based on general v(t)} cannot be built on the regenerated sequence $\{X_k^{\star}\}_{k\in M(\tau_{\epsilon,\delta})}$, because $\tau_{\epsilon,\delta}$ needs to be an $(\mathcal{F}_{m(t)})_{t\in\mathbb{N}_0}$-stopping time, generated by the original sequence $\{X_k\}_{k\in \mathbb{N}}$ alone.
On the basis of the stopping time \eqref{stopping based on general v(t)}, we construct a stopping rule as follows.

\begin{algorithm}[H]
\KwIn{Precision $\epsilon$, error probability $\delta$, random seed $\omega$ and batch indices $\{m(t)\}_{t\in \mathbb{N}_0}$}
\KwOut{$\tau_{\epsilon, \delta}(\omega)$ and $\mu_{\star}(\tau_{\epsilon,\delta})(\omega)$}
 {$t\leftarrow 1$\;}
 
 \While{$2(1-\Phi( \epsilon\sqrt{|M(t)|}/(v(t)+a(t))))(\omega)> \delta$}{
  $t\leftarrow t+1$\;}
  Set $\tau_{\epsilon, \delta}(\omega)\leftarrow t$ and compute $\mu_{\star}(\tau_{\epsilon,\delta})(\omega)$.
 \label{algorithm v(t)}
 \caption{Stopping rule based on a more implementable batch variance}
\end{algorithm}

\vspace{1em}
In a similar manner to \cite{YH1965}, the ``asymptotic'' relevance of the criterion \eqref{stopping based on general v(t)} can be justified in a way that the precision $\epsilon$ and the error probability $\delta$ are adjusted after the sequence $\{X_k\}_{k\in \mathbb{N}}$ of infinite length (which is certainly impossible to keep in hand) has been generated.
We note that the almost sure convergence $v^2(t)/v^2_0(t)\to 1$ can be weakened to the weak convergence $\{v^2(\lfloor t/\epsilon\rfloor)/v^2_0(\lfloor t/\epsilon\rfloor):\,t\in (0,+\infty)\}\to 1$ in $\mathbb{D}((0,+\infty);\mathbb{R})$ as $\epsilon \to 0+$ for a version of Proposition \ref{proposition asymptotic validity 1} based on the functional weak law of large numbers, while it cannot be weakened further to $v^2(t)/v^2_0(t)\stackrel{\mathbb{P}_0}{\to}1$ since then the corresponding version of Part (ii) would not hold anymore (see, for instance, \cite{gut2009}).

\begin{proposition}\label{proposition asymptotic validity 1}
Suppose $\lim_{t\to +\infty}v^2(t)/v^2_0(t)=1$, $\mathbb{P}_0$-$a.s.$, in addition to the conditions imposed in Theorem \ref{theorem base MCLT}.
%If $\lim_{t\to +\infty}\mathbb{P}_0(|\mu(t)-\mu|=\epsilon)=0$, 

\noindent (i) It holds $\mathbb{P}_0$-$a.s.$ that $\tau_{\epsilon,\delta}\to +\infty$ as $\epsilon\to 0+$ or $\delta\to 0+$.

\noindent For (ii) and (iii), fix $\delta\in (0,1)$.

\noindent (ii) It holds $\mathbb{P}_0$-$a.s.$ that $2(1-\Phi(\epsilon\sqrt{|M(\tau_{\epsilon,\delta})|}/(v(\tau_{\epsilon,\delta})+a(\tau_{\epsilon,\delta})))\to \delta$,
%and $ \mathbb{P}_0(|\mu( \tau^a_{\epsilon,\delta})-\mu|>\epsilon)\to \delta$, 
as $\epsilon\to 0+$.

\noindent (iii) If $v^2_0(t)/|M(t)| \sim c t^{-2q}$ as $t\to +\infty$ for some positive $c$ and $q$, then it holds under $\mathbb{P}_0$ that $\epsilon^{1/q} \tau_{\epsilon,\delta} \to (\sqrt{c}\Phi^{-1}(1-\delta/2))^{1/q}$ $a.s.$, and $\epsilon^{-1}(\mu_{\star}(\tau_{\epsilon,\delta})-\mu)\stackrel{\mathcal{L}}{\to} \mathcal{N}(0,1/(\Phi^{-1}(1-\delta/2))^2)$, as $\epsilon\to 0+$.
\end{proposition}

In relation to Proposition \ref{proposition asymptotic validity 1} (ii), we stress that $\mathbb{P}_0(|\mu( \tau_{\epsilon,\delta})-\mu|>\epsilon)\not\to \delta$ in general, as the random variable $\mu(\tau_{\epsilon,\delta})$ is not identical in law to the empirical batch mean $\mu(t)$ at a usual deterministic batch $t$ in the presence of the stopping time $\tau_{\epsilon,\delta}$.
Conversely, the conditional standard deviation $s(\tau_{\epsilon,\delta})$ at the stopping batch $\tau_{\epsilon,\delta}$ cannot be replaced by its regenerated version, that is, $v_{\star}^2(\tau_{\epsilon,\delta})=|M(\tau_{\epsilon,\delta})|^{-1}\sum_{k\in M(\tau_{\epsilon,\delta})} {\rm Var}_{k-1}^{\star}(X^{\star}_k)$, since the regenerated batch $\{X^{\star}_k\}_{k\in M(\tau)}$ is conditionally independent of the stopping batch $\tau_{\epsilon,\delta}$ on the stopping-time $\sigma$-field $\mathcal{F}_{m(\tau_{\epsilon,\delta})}$.

With Proposition \ref{proposition asymptotic validity 1} being presented, of practical interest is the non-asymptotic regime \cite{BHST, FLYA2012} in which one generates a segment $\{X_k\}_{k\in M(t)}$ of the sequence, one batch at a time forward until, and only until, the criterion \eqref{stopping based on general v(t)} is satisfied, with precision $\epsilon$ and error probability $\delta$ fixed.

%\textcolor{red}{For later use, we take the following note: Define by $\Delta=\sup_{x}|F(x)-G(x)|$, and $Z\sim\mathcal{N}(0,1)$. Then, for a triangular array $\{X_{ni}\}_{n\in\mathbb{N}, i\in\{1,2,\cdots,n\}}$. According to \cite{sydney}[Page 30], we know that $\Delta(S_{n}, Z) \leqq c(\delta) L_{n \delta}^{1 /(3+\delta)}$, where $c(\delta)<4 \cdot 7$ and $c(1) \leqq 2 \cdot 23$. 
%Moreover, we define the distance $\Lambda(X, Y):=\inf \{h: F(x-h)-h \leqq G(x) \leqq F(x+h)+h, \forall x\}$,
%\[
%\Lambda\left(\sum_{1}^{n+1} W_{n k}, S_{n}\right) \leqq v(\delta)\left[L_{n \delta}+\frac{1}{\sigma^2}\mathbb{E}\left|v_1^2(t)-\sigma^2\right|^{1+c/ 2}\right]^{1 /(3+\delta)},
%\]
%and
%\[
%v(\delta)=\left[2^{1+1 / 2 \delta} \pi^{-1 / 2} \Gamma\left(\frac{3+\delta}{2}\right)\right]^{1 /(3+\delta)}+\left[2 \cdot 10^{2+\delta}\left(1+2^{1 / 2 \delta}\right)\right]^{1 /(3+\delta)}.
%\]
%Finally,
%\[
%\Delta\left(S_{n}, N\right) \leqq \gamma(\delta)\left[L_{n \delta}+\mathbb{E}\left|1-\sigma_{n}^2\right|^{1+1 / 2 \delta}\right]^{1 /(3+\delta)},
%\]
%with
%\[
%\gamma(\delta)=\left(1+\frac{1}{\sqrt{2 \pi}}\right)(c(\delta)+v(\delta))<21.6.
%\]
%}

\subsubsection{Conditional batch variances}
\label{subsection Stopping rule based on conditional batch variance}

A candidate for the sequence $\{v(t)\}_{t\in\mathbb{N}}$ is based on the conditional batch variance and standard deviation, defined as 
\begin{equation}\label{conditional batch variance}
v_1^2(t):=\frac{1}{|M(t)|}\sum_{k\in M(t)} {\rm Var}_{k-1}\left(X_k\right)
%=\frac{1}{|M(t)|}\sum_{k\in M(t)} \mathbb{E}_{k-1}\left[\left(X_k-\mathbb{E}_{k-1}[X_k]\right)^2 \right]
=\frac{1}{|M(t)|}\sum_{k\in M(t)} \mathbb{E}_{k-1}\left[X_k^2 \right]-\mu^2,
\end{equation}
and $v_1(t):=(v_1^2(t))^{1/2}$ for $t\in \mathbb{N}$, where we have applied the base assumption \eqref{base assumption} for the evaluation of the conditional variance.
The conditional batch variance here is expected to behave in a somewhat similar manner to the theoretical one in the sense that it is unbiased for the theoretical batch variance \eqref{theoretical batch variance}, that is, $\mathbb{E}_0[v_1^2(t)]=v^2_0(t)$ for all $t\in \mathbb{N}$.
Based on the availability of the sample $\{X_k\}_{k\in M(t)}$ alone, the conditional batch variance and standard deviation \eqref{conditional batch variance} are, in principle, more implementable than the theoretical ones \eqref{theoretical batch variance}, whereas the reality is that neither is generally implementable, with a few exceptions (such as the one examined in Section \ref{section ARCH(1)}, where the conditional one \eqref{conditional batch variance} happens to be available in an implementable form).
%Still, we construct the related stopping rule \eqref{stopping based on general v(t)} and justify its asymptotic relevance (Proposition \ref{proposition asymptotic validity 1}) for the sake of consistency and completeness.

For later use, we amend the criterion \eqref{stopping based on general v(t)} with $v(t)$ replaced by the conditional counterpart $v_1(t)$, as follows:  
\begin{equation}\label{stopping based on s(t)}
  \tau^a_{\epsilon,\delta}:=\min\left\{t\in \mathbb{N}:\,2\left(1-\Phi\left(\frac{\epsilon \sqrt{|M(t)|}}{v_1(t)+a(t)}\right)\right)\le \delta\right\},
\end{equation}
as a stopping batch index.
 
%\color{red}
%Let $X_k|_{\mathcal{F}_{k-1}}\sim U\{0,-\epsilon X_{k-1},+\epsilon X_{k-1}\}$.
%Then, $\mathbb{E}_{k-1}[X_k]=0$ and ${\rm Var}_{k-1}(X_k)=(2/3)\epsilon^2X_{k-1}^2$.

%\color{black}

%\subsection{The stopping rule based on empirical batch variance}
%\label{subsection Stopping rule based on empirical batch variance}

\subsubsection{Empirical batch variances}
\label{subsection Stopping rule based on empirical batch variance}

A more realistic candidate for the sequence $\{v(t)\}_{t\in \mathbb{N}}$ can be constructed based on the empirical batch variance and standard deviation, respectively, defined as 
\begin{equation}\label{empirical batch variance}
v_2^2(t):=\frac{1}{|M(t)|-1} \sum_{k \in M(t)} \left(X_k-\mu(t)\right)^2,\quad v_2(t):=\left(v_2^2(t)\right)^{1/2},\quad t\in \mathbb{N},
\end{equation}
where $\mu(t)$ is the (implementable) empirical batch mean \eqref{empirical batch mean}, rather than the (unknown) constant $\mu$ in the base assumption \eqref{base assumption}.
The relevance is that as opposed to the conditional batch variance and standard deviation \eqref{conditional batch variance}, the empirical batch variance $v_2^2(t)$ and standard deviation $v_2(t)$ here are implementable as soon as the sample $\{X_k\}_{k\in M(t)}$ is given, yet with unbiasedness $\mathbb{E}_0[v_2^2(t)]=v^2_0(t)$ holding true.
An empirical version of the martingale central limit theorem can be constructed on the basis of the empirical standard deviation \eqref{empirical batch variance}, as follows. 
%, so as to replace the original version \eqref{first MCLT} for implementation purposes.

\begin{theorem}\label{theorem empirical MCLT}
In addition to the conditions imposed in Theorem \ref{theorem base MCLT}, let the following two conditions hold: $\limsup_{k\to +\infty} \mathbb{E}_0[|X_k|^4]<+\infty$ and $\lim_{t\to +\infty}|M(t)|^{-1}\sum_{k\in M(t)}(X_k^2-\mathbb{E}_{k-1}[X_k^2])= 0$, $\mathbb{P}_0\mbox{-}a.s.$
% \begin{equation}%\label{quaratic condition}
% \frac{1}{|M(t)|}\sum_{k\in M(t)}\left(X_k^2-\mathbb{E}_{k-1}\left[X_k^2\right]\right)\to 0,
% %\limsup_{t\to +\infty} \sum_{k\in M(t)} \frac{1}{(k-m(t-1))^2}\mathbb{E}_{k-1}\left[|X_k-\mu|^{4} \right]<+\infty, 
% \quad \mathbb{P}_0\mbox{-}a.s.
% \end{equation}
It then holds that $v_2^2(t)/v_1^2(t)\stackrel{\mathbb{P}_0}{\to}1$.
% and, under $\mathbb{P}_0$,
% \begin{equation}\label{second MCLT}
% \sqrt{|M(t)|} \frac{\mu(t)-\mu}{v_2(t)}\stackrel{\mathcal{L}}{\to}\mathcal{N}(0,1),
% \end{equation}	
% as $t \to+\infty$.
\end{theorem}

%We note that the denominator in the condition \eqref{quaratic condition} is not $k^2$ but $(k-m(t-1))^2$, due to the arrangement $M(t)=\{m(t-1)+1,\cdots,m(t)\}$, just as in Assumption \ref{standing assumption 2}.
%The condition \eqref{quaratic condition} holds true even when the conditional centered fourth moment explodes as long as $\mathbb{E}_{k-1}[(X_k-\mu)^4]=o(k-m(t-1))$ as $k\to +\infty$ (under constraint $k\in M(t)$).
%In view of the empirical version of the martingale central limit theorem \eqref{second MCLT}, 
For later use, we amend the criterion \eqref{stopping based on general v(t)} with $v(t)$ replaced by the empirical counterpart $v_2(t)$, as follows:  
\begin{equation}\label{stopping based on sigma(t)}
  \tau^b_{\epsilon,\delta}:=\min\left\{t\in \mathbb{N}:\,2\left(1-\Phi\left(\frac{\epsilon \sqrt{|M(t)|}}{v_2(t)+a(t)}\right)\right)\le \delta\right\},
\end{equation}
as a stopping batch index.

\section{Numerical illustrations}\label{section numerical illustations}

%We let $\{\xi_k\}_{k\in\mathbb{N}}$ be a sequence such that $\xi_k:=X_k-\mu$ for all $k\in \mathbb{N}$. 
%The stochastic process $\{\xi_k: k\in\mathbb{N}\}$ is then a martingale difference sequence with respect to the filtration $(\mathcal{F}_k)_{k\in\mathbb{N}_0}$ satisfying
%$\mathbb{E}_{k-1}[\xi_k ]=0$. 

%\begin{itemize}
%\item
%Case $1$: the sequence converges on its own, that is, $\mathbb{E}_{k-1}[\xi_k^2]\to\sigma^2$ almost surely. 
%\item
%Case $2$: only the average of the sequence converges, that is
%\[
%\frac{1}{|M(t)|} \sum_{k\in M(t)} \mathbb{E}_{k-1}\left[\xi_k^2\right] \to\sigma^2, \quad \mathbb{P}_0 - \text{a.s}.
%\] 
%\item
%Case $3$: for all $k$ sufficiently large, it holds that $| \mathbb{E}_{k-1}[\xi_k^2]-\sigma^2| < \epsilon$. That is, $\mathbb{E}_{k-1}[\xi_k^2]$ may oscillate for good.
%\end{itemize}
%Clearly, possibility $1$ implies possibility $2$.
%\setlength{\tabcolsep}{30pt}
%\begin{table}[htbp]
%\begin{center}
%\begin{tabular}{|c|c|c|c|c|}
%\hline
%Case & \multicolumn{4}{p{4.5cm}|}{\centering Numerical example}  \\
%\cline{2-5} & \multicolumn{1}{c|}{Section $6.1$} & \multicolumn{1}{c|}{Section $6.2$} & \multicolumn{1}{c|}{Section $6.3$} & \multicolumn{1}{c|}{Section $6.4$}\\ \hline
%Case 1 &  & $\surd$ & & $\surd$  \\
%Case 2 & $\surd$ &  &  &   \\
%Case 3 &  &  & $\surd$ &  \\
%\hline
%\end{tabular}
%\end{center}
%\caption{which example is in which case}
%\label{comparisons of examples}
%\end{table}

In this section, we present numerical examples to demonstrate the procedure and illustrate the effectiveness and the strong potential of the proposed stopping rules in making a satisfactory stop of the sequential experiment.
To this end, we first summarize the developed stopping rules.

For precision $\epsilon\in (0,+\infty)$ and error probability $\delta\in (0,1)$, one stops the experiment at a batch $\tau_{\epsilon,\delta}$ in which the probability of the undesirable event $\{|\mu(t)-\mu|>\epsilon\}$ is satisfactorily small for the first time, that is, $\tau_{\epsilon,\delta}=\min_{t\in\mathbb{N}}\{\mathbb{P}_{m(t-1)}(|\mu(t)-\mu|>\epsilon)\le \delta\}$.
Depending on how to approximate the error probability $\mathbb{P}_{m(t-1)}(|\mu(t)-\mu|>\epsilon)$ in the criterion, we have developed the following two stopping rules:
\begin{equation}\label{collection of taus}
\tau_{\epsilon,\delta} \leftarrow \begin{dcases}
\tau^a_{\epsilon,\delta}=\min\left\{t\in\mathbb{N}:2 \left( 1 - \Phi ( \epsilon \sqrt{|M(t)|}/(v_1(t)+a(t)) ) \right) \le \delta\right\},& \text{as of } \eqref{stopping based on s(t)},\\
\tau^b_{\epsilon,\delta}=\min\left\{t\in\mathbb{N}:\,2 \left( 1 - \Phi ( \epsilon\sqrt{|M(t)|} /(v_2(t)+a(t)) ) \right)\le  \delta\right\},& \text{as of } \eqref{stopping based on sigma(t)},
%\\
%\tau_{\epsilon,\delta}^c=\argmin_{t\in\mathbb{N}}\left\{2 \left( 1 - \Phi( \epsilon\sqrt{|M(t)|}/(v_2(t)+a(t)))+\gamma_q (\ell_q(t))^{1/(3+q)} \right)\le  \delta\right\},& \text{as in } \eqref{empirical criterion based on berry-esseen},
\end{dcases}
\end{equation}
where $v_1(t)$ and $v_2(t)$ are the conditional and empirical batch standard deviations, defined respectively in \eqref{conditional batch variance} and \eqref{empirical batch variance}.
To assess the proposed stopping rules, we take the following steps.

\begin{description}
\setlength{\parskip}{0cm}
\setlength{\itemsep}{0cm}

\item[(i)] Fix the precision $\epsilon \in (0,+\infty)$ and the error probability $\delta \in (0,1)$.

\item[(ii)] Fix the batches $\{m(t)\}_{t\in \mathbb{N}}$, the number of iid experiments $n$, and then $n$ distinct random seeds $(\omega_1,\cdots,\omega_n)(=:\omega)$.

\item[(iii)] For every random seed $\omega_k$, obtain the stopping batch index $\tau_{\epsilon,\delta}(\omega_k)$ in accordance with either \eqref{stopping based on s(t)} or \eqref{stopping based on sigma(t)},
%or \eqref{empirical criterion based on berry-esseen}, 
and compute the indicator $\mathbbm{1}(|\mu_{\star}(\tau_{\epsilon,\delta})(\omega_k)-\mu|\leq \epsilon)$ over the $\tau_{\epsilon,\delta}(\omega_k)$th batch in accordance with \eqref{def of mu*}.

\item[(iv)] Compute an empirical error probability $p_{\epsilon,\delta}(\omega)$, the so-called reliability $R_{\epsilon,\delta}(\omega)$, and the complexity ${\rm CM}_{\epsilon,\delta}(\omega)$, respectively, as follows:
\begin{gather}\label{reliability complexity}
p_{\epsilon,\delta}(\omega):= \frac{1}{n} \sum_{k=1}^n \mathbbm{1}\left(|\mu_{\star}(\tau_{\epsilon,\delta})(\omega_k)-\mu|\leq \epsilon\right),\quad R_{\epsilon,\delta}(\omega):= \frac{p_{\epsilon,\delta}(\omega)}{1-\delta},\nonumber\\
\text{CM}_{\epsilon,\delta}(\omega) := \frac{1}{n}\sum_{k=1}^{n} \left[(\tau_{\epsilon,\delta}(\omega_k))^{\ell} - ({\tau_{\epsilon,\delta}(\omega_k)}-1)^{\ell}\right].
\end{gather}
\end{description}

In accordance with \cite{BHST}, we consider a stopping rule with its reliability $R_{\epsilon,\delta}(\omega)$ strictly less than $1$ unreliable.
In addition, the effectiveness of a stopping rule improves as its complexity $\text{CM}_{\epsilon,\delta}(\omega)$ decreases at  
%\textcolor{red}{since the complexity represents the average number of iterations required to achieve 
the specified precision and error probability.
Throughout the following experiments, we fix $n=5\times 10^3$ to ensure enough realizations.
We note that the parameter $\ell$ in the complexity $\text{CM}_{\epsilon,\delta}(\omega)$ controls how severely we penalize large stopping times and is suppressed from its notation because we fix $\ell=5$ throughout.

%Clearly, the parameter $\ell$ does not affect the convergence properties of the stopping rule itself.
%, as these are guaranteed by our theoretical results.

With suitable constants $0<a_1<a_2<+\infty$ and $0<b_1<b_2<1$, we split the rectangle $[a_1, a_2]\times [b_1, b_2]$ into an logarithmically spaced grid of $10\times 10$ points, denoted by $\{(\epsilon_k,\delta_k)\}_{k\in \{1,\cdots,10\times 10\}}$.
On each of these $10\times 10$ combinations of precision and error probability, we compute the reliability $R_{\epsilon_k,\delta_k}(\omega)$ and the complexity ${\rm CM}_{\epsilon_k,\delta_k}(\omega)$, with the random seeds $\omega$ fixed throughout.

In both examples below (Sections \ref{section ARCH(1)} and \ref{section adaptive control variates}), the theoretical variance $v_0^2(t)$ converges to a positive constant, so that $|M(t+1)| \sim |M(t)|$ to satisfy the technical condition $v_0^2(t+1)/|M(t+1)| \sim v_0^2(t)/|M(t)|$ stated in Assumption \ref{standing assumption 2}.
Accordingly, we set the batch sizes as $M(t)=\{(t-1)^5+1,\cdots,t^5\}$ so that $|M(t)|\propto t^4$, striking a practical balance.
Slower growth (for instance, $|M(t)| \propto t^2$) reduces excess sampling but requires more frequent checks, whereas faster growth (for instance, $|M(t)| \propto t^6$) decreases the number of checks but risks overshooting the required sample size.

%\subsection{Biased random walk}\label{section random walk}

\subsection{ARCH(1)}\label{section ARCH(1)}

We begin our numerical illustrations with an ARCH(1) model, which provides an ideal test case for our stopping rules due to its non-trivial dependence structure and heavy-tailed behavior, as well as the straightforward implementation of the theoretical and conditional batch variances.
This setting poses natural challenges for sequential sampling methods, as the underlying dynamics combine temporal dependence through the ARCH structure with potentially extreme observations from the $t$-distribution.

Now, consider an ARCH(1) sequence $\{X_k\}_{k\in\mathbb{N}}$, which is generated recursively by $X_k=(\beta+\alpha X_{k-1}^2)^{1/2} V_k$, with $X_0=0$, $\alpha\in (0,1)$ and $\beta\in (0,+\infty)$.
Here, $\{V_k\}_{k\in\mathbb{N}}$ denotes a sequence of iid random variables following a scaled $t$-distribution with $\nu$ degrees of freedom ($\nu\in\{5,6,\cdots\}$), normalized to have unit variance, with its probability density function given as 
\[
 \frac{d}{dx}\mathbb{P}(V_1\le x)=\frac{\Gamma((\nu+1)/2)}{\sqrt{(\nu-2)\pi}\Gamma(\nu/2)}\left(1+\frac{x^2}{\nu-2}\right)^{-(\nu+1)/2}, \quad x\in\mathbb{R}.
\]
The theoretical and conditional batch variances for this model are given by
\[
 v^2_0(t) =\frac{\beta}{1-\alpha}\left(1-\frac{\alpha^{m(t-1)+1}-\alpha^{m(t)}}{2}\right),\quad 
 v_1^2(t)=\beta+\frac{\alpha}{|M(t)|}\sum_{k\in M(t)}X_{k-1}^2.
\]
The filtration $(\mathcal{F}_k)_{k\in\mathbb{N}_0}$ is constructed as $\mathcal{F}_k:=\sigma(\mathcal{N})\vee\sigma(V_1,\cdots,V_k)$ for all $k\in\mathbb{N}_0$, while the branched filtration $(\mathcal{F}_{k-1}^{\star})_{k\in M(\tau)}$ is constructed in the way of $\mathcal{F}_{k-1}^{\star}=\mathcal{F}_{m(\tau)}\vee \sigma(V_{m(\tau-1)+1}^{\star},\cdots,V_{k-1}^{\star})$ for all $k\in M(\tau)$.
We note that $\sigma$ here denotes the $\sigma$-field, rather than the empirical batch variance \eqref{empirical batch variance}.

In what follows, we fix $\alpha=0.03$, $\beta=0.3$, and $\nu=6$ degrees of freedom.
As we observe from Figure \ref{fig:variance_estimates}, after some initial fluctuations due to smaller batch sizes, the theoretical, conditional, and empirical variances exhibit similar behavior, with all those measures tracking each other closely.
This validates the use of the empirical variance $v_2^2(t)$ in the stopping rule, as it succeeds in approximating the theoretical variance, not to mention the conditional variance, in the long run.

\begin{figure}[H]
    \centering
    \includegraphics[width=0.6\textwidth]{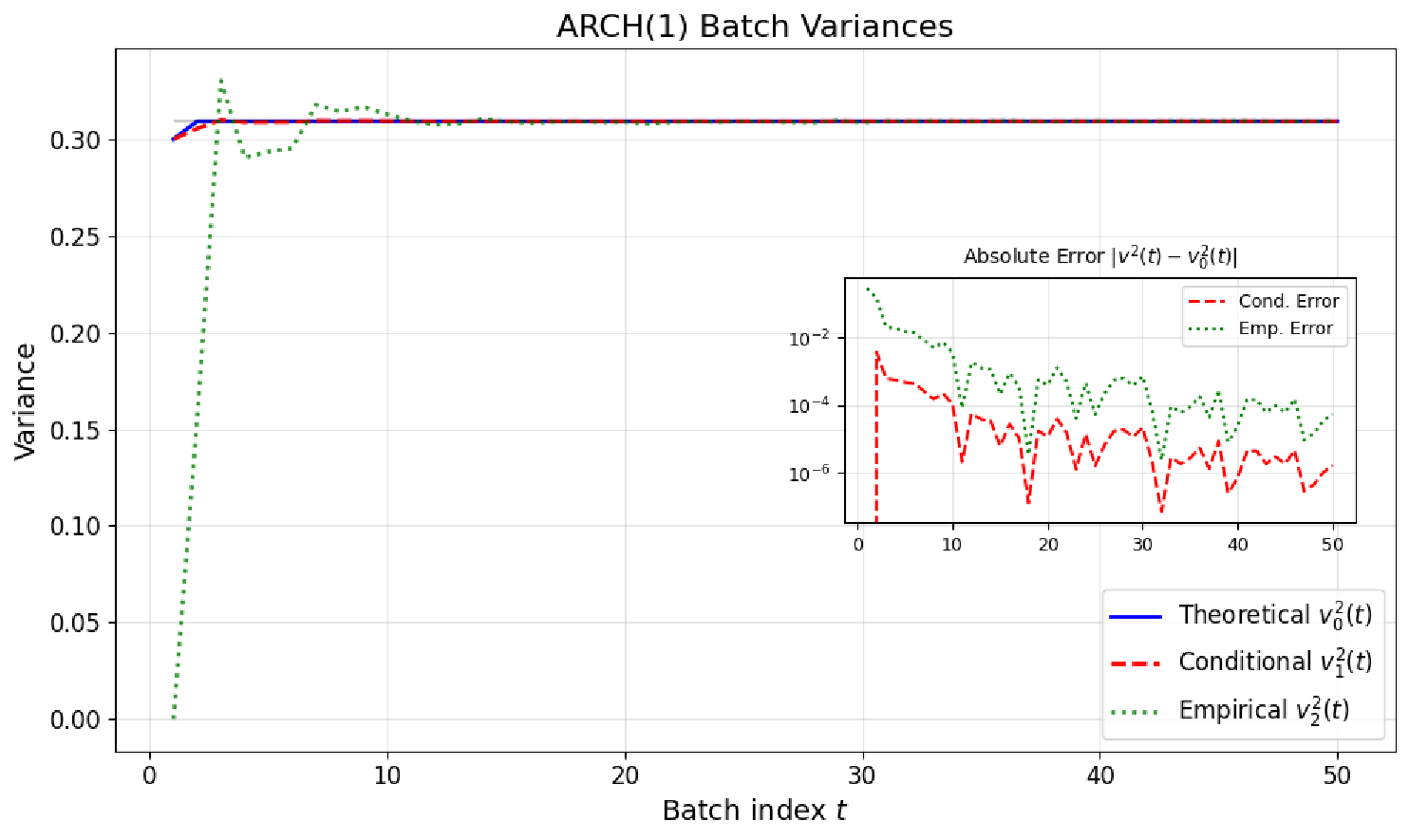}
    \caption{Typical realization of theoretical batch variance $v_0^2(t)$ (solid blue), conditional batch variance $v_1^2(t)$ (dashed red), and empirical batch variance $v_2^2(t)$ (dotted green) of the ARCH(1) process under consideration. The inset displays the absolute estimation errors $|v_1^2(t) - v_0^2(t)|$ (dashed red) and $|v_2^2(t) - v_0^2(t)|$ (dotted green) on a logarithmic scale, illustrating the convergence rates to the theoretical variance.}
    \label{fig:variance_estimates}
\end{figure}

We implement the empirical variance stopping rule $\tau^b_{\epsilon,\delta}$ based on the empirical batch variance $v_2^2(t)$ as of \eqref{collection of taus}, as it can be readily available unlike the theoretical and conditional variances.
Here, we set $a(t) = 1/t$, which helps to compensate the severe underestimation of the empirical variance $v_2^2(t)$ during early batches, such as $v_2^2(1)\approx 0$ observed in Figure \ref{fig:variance_estimates}.
The choice of $a(t)=1/t$ satisfies its theoretical requirement of monotonic vanishing and reduces the risk of premature stopping.

We assess its performance by examining reliability and computational complexity \eqref{reliability complexity} across various precision ($\epsilon$) and error probability ($\delta$) values. For each $(\epsilon_k,\delta_k)$ in the specified grid, we perform independent simulations to compute $\tau^b_{\epsilon_k,\delta_k}(\omega_j)$, generate new batches for the regenerated mean $\mu_{\star}(\tau_{\epsilon_k,\delta_k})(\omega_j)$, and evaluate success through $\mathbbm{1}(|\mu_{\star}(\tau_{\epsilon_k,\delta_k})(\omega_j)-\mu|\leq \epsilon_k)$.

\begin{figure}[H]
    \centering
    \includegraphics[width=\textwidth]{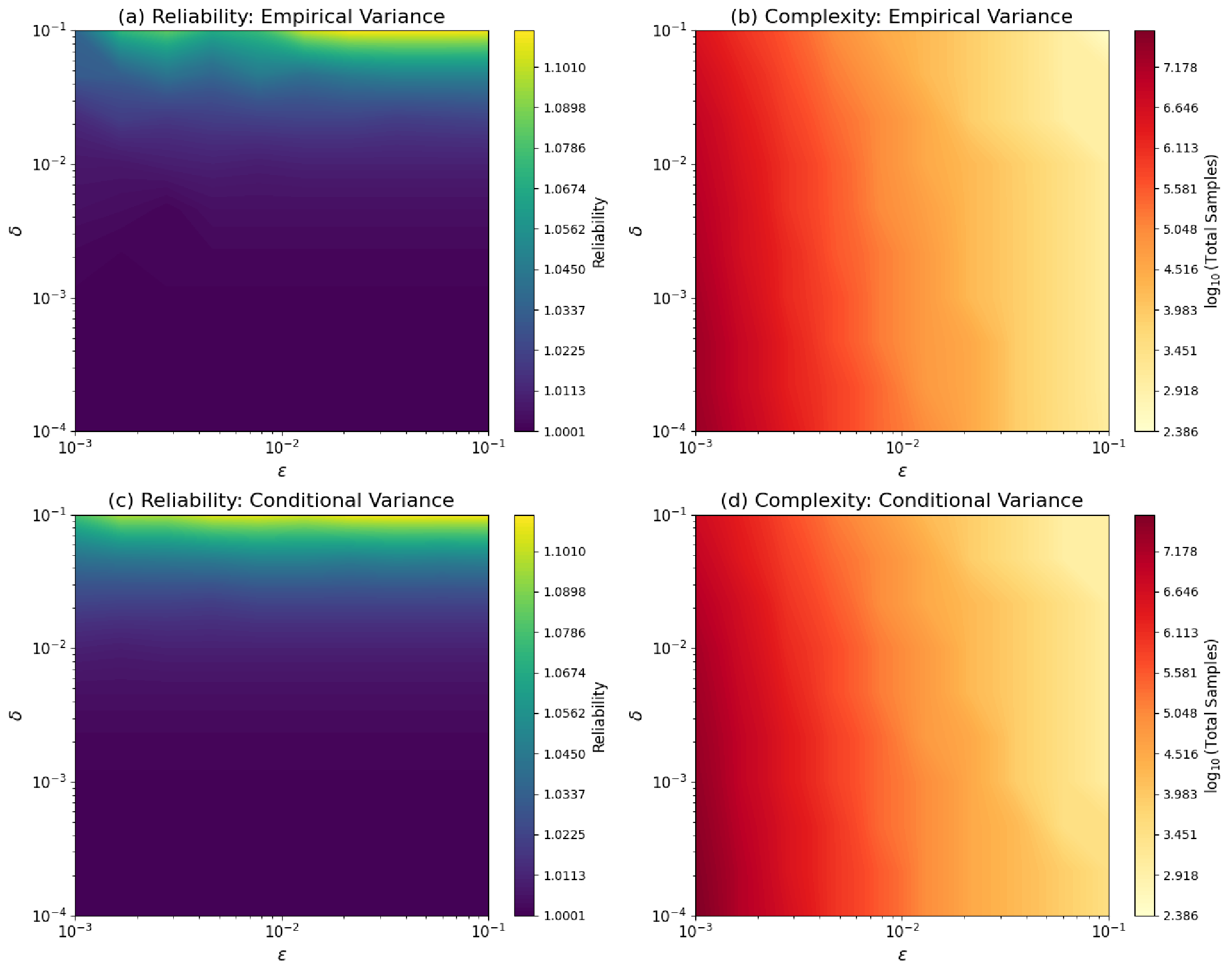} % Replace with actual path to your image
    \caption{The left column ((a) and (c)) represents the reliability $R_{\epsilon,\delta}$, while the right column ((b) and (d)) shows the complexity ${\rm CM}_{\epsilon,\delta}$ on a $\log_{10}$ scale.
    The top row ((a) and (b)) corresponds to the empirical-variance stopping rule ($\tau^b_{\epsilon,\delta}$), whereas the bottom row ((c) and (d)) corresponds to the conditional-variance stopping rule ($\tau^a_{\epsilon,\delta}$).}
    \label{fig:complexity_reliability}
\end{figure}

\begin{table}[H]
\centering
\renewcommand{\arraystretch}{1.2} % Adds vertical breathing room
\caption{Performance metrics comparison between the conditional-variance ($\tau^a_{\epsilon,\delta}$) and empirical-variance ($\tau^b_{\epsilon,\delta}$) stopping rules.}
\label{tab:stopping_rule_metrics}
\begin{tabular}{lrrr}
\toprule
\textbf{Metric} & \textbf{Mean} & \textbf{Min} & \textbf{Max} \\
\midrule
\multicolumn{4}{l}{Reliability $R_{\epsilon,\delta}$} \\
\hspace{3mm} empirical (Figure \ref{fig:complexity_reliability} (a))   & 1.0165 & 1.0000 & 1.1111 \\
\hspace{3mm} conditional (Figure \ref{fig:complexity_reliability} (c)) & 1.0191 & 1.0001 & 1.1111 \\
\addlinespace % Adds a small gap between groups
\multicolumn{4}{l}{Complexity $\rm CM_{\epsilon,\delta}$} \\
\hspace{3mm} empirical (Figure \ref{fig:complexity_reliability} (b))   & 1,965,883.60 & 243.00 & 24,300,000.00 \\
\hspace{3mm} conditional (Figure \ref{fig:complexity_reliability} (d)) & 3,339,877.00 & 1,024.00 & 45,435,424.00 \\
\bottomrule
\end{tabular}
\end{table}

In Figure \ref{fig:complexity_reliability} and Table \ref{tab:stopping_rule_metrics} above, we present our comprehensive numerical findings on the performance of the proposed stopping rules.
The plots of the reliability $R_{\epsilon,\delta}$ (Figure \ref{fig:complexity_reliability} (a) and (c)) demonstrate that our stopping rules consistently achieve the reliability staying strictly greater than $1$ across all parameter combinations.
As confirmed in Table \ref{tab:stopping_rule_metrics}, the reliability ranges from 1.000 to 1.111, satisfying the reliability criterion set up in \cite{BHST}.
Notably, the conditional-variance stopping rule $\tau^a_{\epsilon,\delta}$ exhibits a slightly higher mean (1.0191) compared to the empirical-variance stopping rule $\tau^b_{\epsilon,\delta}$ (1.0165), highlighting the conservative nature of using the true model parameters.
The reliability increases systematically with larger values of $\epsilon$ and $\delta$, reflecting more conservative behavior in less demanding scenarios while maintaining reliability even in the most challenging case with $\epsilon=\delta=10^{-3}$.

The heat maps of the complexity ${\rm CM}_{\epsilon,\delta}$ (Figure \ref{fig:complexity_reliability} (b) and (d)) reveal a pattern of computational cost across the parameter space.
The logarithmic scale clearly shows how the complexity $\rm CM_{\epsilon,\delta}$ increases as both $\epsilon$ and $\delta$ decrease.
The computational burden varies substantially, as detailed in Table \ref{tab:stopping_rule_metrics}, ranging from $2.43\times 10^2$ for the empirical-variance stopping rule with larger values of $\epsilon$ and $\delta$ to $4.54\times 10^7$ for the conditional-variance stopping rule with $\epsilon=\delta=10^{-3}$.
Directly comparing the methods, the empirical-variance stopping rule requires approximately $2.0\times 10^6$ samples on average versus $3.3\times 10^6$ for the conditional-variance stopping rule, demonstrating that the empirical approach achieves significant computational efficiency without violating the reliability requirement.
This scaling aligns with our theoretical predictions and provides practical guidance for computational requirements.

Let us stress that smooth transitions in both reliability and complexity across the parameter space indicate the stability and robustness of the proposed stopping rules.
The consistent achievement of the requirement ($R_{\epsilon,\delta}>1$), coupled with theoretically predicted complexity scaling, provides strong empirical validation of our methodology, particularly notable given the challenging features of temporal dependence and heavy-tailed distributions in the ARCH(1) setting.

\subsection{Adaptive control variates}\label{section adaptive control variates}

%The variance reduction method of control variates is effective particularly when its control variate has a high correlation, whether positive or negative, with the estimator of interest.
In this section, we examine the problem setting described in Examples \ref{examples of problem settings} (c) and \ref{examples of regenerating} (c).
To recall it, suppose, for $\Psi: (0,1)^d \to \mathbb{R}$, we wish to evaluate the integral $\mu=\int_{(0,1)^d} \Psi({\bf u}) d{\bf u}$, where the random element here is modeled via the uniform law on the unit hypercube $(0,1)^d$, which can be considered general enough thanks to the inverse transform random sampling method.
Here, without affecting the value of the integral, one can introduce a parameter ${\bm \theta}\in \mathbb{R}^d$  in a linear manner, as follows
%Given the control variates parameterized by ${\bm \theta}$ and a centered random vector $C: \Omega \to \mathbb{R}^{d}$, we have the following expression
\begin{equation}\label{control variates}
\mu=\int_{(0,1)^d}\Psi({\bf u}) d{\bf u}=\int_{(0,1)^d}\left(\Psi({\bf u})+\left\langle{\bm \theta}, {\bf u}-\mathbbm{1}_d/2\right\rangle\right)d{\bf u}, %=:\int_{(0,1)^d} R({\bf u};{\bm \theta})d{\bf u},
\end{equation}
under the quartic integrability condition $\int_{(0,1)^d}|\Psi({\bf u})|^4d{\bf u}<+\infty$ for technical reasons (Section \ref{appendix CV}).
Note that the control variate $U-\mathbbm{1}_d/2$ is centered with variance-covariance matrix $\int_{(0,1)^d}({\bf u}-\mathbbm{1}_d/2)^{\otimes 2}d{\bf u}=\mathbb{I}_d/12$, that is, the diagonal matrix of order $d$ with all diagonal elements of $1/12$.
Although one may introduce nonlinear control variates, we do not go in that direction in the present paper.
Now, define $V({\bm \theta}):=\int_{(0,1)^d}(\Psi({\bf u})+\langle{\bm \theta}, {\bf u}-\mathbbm{1}_d/2\rangle)^2 d{\bf u}$, which represents the second moment of the rightmost expression in \eqref{control variates}, parameterized by ${\bm \theta}$.
The estimator variance $V({\bm \theta})-\mu^2$ is, due to its quadratic structure in the parameter ${\bm \theta}$, minimized uniquely at  
\begin{equation}\label{optimizer of V}
{\bm \theta}^*:=\argmin_{{\bm \theta} \in \mathbb{R}^d} V({\bm \theta})= -12\int_{(0,1)^d}\Psi({\bf u})({\bf u}-\mathbbm{1}_d/2)d{\bf u},
\end{equation}
with the minimized second moment
%We are interested in the unique optimizer ${\bm \theta}^{*}$ to minimize $V({\bm \theta}^{*}).$ 
%Due to the quadratic structure of $V({\bm \theta})$ in ${\bm \theta}$, we have the following
%\[
%\nabla_{{\bm \theta}} V({\bm \theta})=2\int_{(0,1)^d}C(\omega)\Psi({\bf u})\mathbb{P}(d\omega)+2\int_{(0,1)^d}\left(C(\omega)\right)^{\otimes 2}d{\bf u}{\bm \theta}=2\int_{(0,1)^d}C(\omega)\Psi({\bf u})d{\bf u}+2\Sigma{\bm \theta}.
%\] 
%It follows readily that 
\begin{equation}\label{minimal second moment}
\min_{{\bm \theta} \in \mathbb{R}^{d}} V({\bm \theta})=V({\bm \theta}^*)=V(0)-\frac{1}{12}\|{\bm \theta}^*\|^2,
%\int_{(0,1)^d}\left(\Psi({\bf u})-\mu\right)^2d{\bf u}
%-\left\langle\int_{(0,1)^d}\Psi({\bf u})C(\omega)d{\bf u}, \Sigma^{-1}\int_{(0,1)^d}\Psi({\bf u})C(\omega)d{\bf u}\right\rangle,
\end{equation}
%where the minimizer, ${\bm \theta}^{*}$, is attained uniquely when
where the first term $V(0)$ represents the crude second moment.
Hence, the minimized second moment \eqref{minimal second moment} is strictly smaller than the crude second moment $V(0)$ unless ${\bm \theta}^*=0_d$, that is, the integrand $\Psi({\bf u})$ is linearly uncorrelated with the underlying random element ${\bf u}$, which is rather absurd from a practical point of view.
%integral $\int_{(0,1)^d}\Psi({\bf u})C(\omega)d{\bf u}$ is exactly the zero vector.
It is known %(for instance, \cite{KAWAI2020123608}) 
that Monte Carlo averaging, albeit no longer iid due to the evolving adaptive sequence $\{{\bm \theta}_k\}_{k\in\mathbb{N}_0}$, can proceed concurrently along with the parameter search, as follows:   
\begin{equation}\label{original CV}
 \frac{1}{n}\sum_{k=1}^n \left(\Psi(U_k)+\left\langle{\bm \theta}_{k-1}, U_k-\mathbbm{1}_d/2\right\rangle\right)\to \mu,\quad {\bm \theta}_n:=-12\frac{1}{n}\sum_{k=1}^n \Psi(U_k)(U_k-\mathbbm{1}_d/2)(\to {\bm \theta}^*),
\end{equation}
with ${\bm \theta}_0:=0_d$, where $\{U_k\}_{k\in\mathbb{N}}$ is a sequence of iid uniform random vectors on $(0,1)^d$.
Note that both convergences in \eqref{original CV} hold true $\mathbb{P}_0$-$a.s.$ as $n\to +\infty$.

In the present context, instead of updating the parameter ${\bm \theta}$ at every iteration as in \eqref{original CV}, computing budget may be saved by updating it only once at the end of each batch (right before moving on to the next batch), as follows:
\begin{equation}\label{thetat CV}
 {\bm \theta}(t):=-12\frac{1}{|M(t)|} \sum_{k\in M(t)} \Psi(U_k)(U_k-\mathbbm{1}_d/2),\quad t\in \mathbb{N},
\end{equation}
with ${\bm \theta}(0)=0_d$.
By representing ${\bm \theta}(t)$ as a function of the batch index, we stress that it is updated only once for each batch. 
It clearly remains true that ${\bm \theta}(t)\to {\bm \theta}^*$, $\mathbb{P}_0$-$a.s.$ as $t\to +\infty$, just as in \eqref{original CV}, because this is still nothing but standard iid Monte Carlo methods for the rightmost expression of \eqref{optimizer of V}.
Here, the filtration $(\mathcal{F}_k)_{k\in\mathbb{N}_0}$ can be set as  $\mathcal{F}_k=\sigma(\mathcal{N})\vee\sigma(U_1,\cdots,U_k)$ for $k\in\mathbb{N}$, as well as the branched one $(\mathcal{F}_{k-1}^{\star})_{k\in M(\tau)}$ by $\mathcal{F}_{k-1}^{\star}=\mathcal{F}_{m(\tau_{\epsilon,\delta})}\vee \sigma(U_{m(\tau_{\epsilon,\delta}-1)+1}^{\star},\cdots,U_{k-1}^{\star})$ for all $k\in M(\tau_{\epsilon,\delta})$ on the iid regenerated batch $\{U_k^{\star}\}_{k\in M(\tau_{\epsilon,\delta})}$, where $\tau_{\epsilon,\delta}$ is a suitable stopping batch as in \eqref{collection of taus}. 
Then, for every $t\in\mathbb{N}$ and $k\in M(t)$, the base formulation (Assumption \ref{base assumption}) holds by setting $X_k=\Psi(U_k)+\left\langle{\bm \theta}(t-1), U_k-\mathbbm{1}_d/2\right\rangle$, as
\begin{align*}
 \mathbb{E}_{k-1}\left[X_k\right]&=\mathbb{E}_{k-1}\left[\Psi(U_k)+\left\langle{\bm \theta}(t-1), U_k-\mathbbm{1}_d/2\right\rangle\right]\\
 &=\mathbb{E}_{k-1}\left[\Psi(U_k)\right]+\left\langle{\bm \theta}(t-1),\mathbb{E}_{k-1}\left[ U_k-\mathbbm{1}_d/2\right]\right\rangle=\mu,
\end{align*}
since ${\bm \theta}(t-1)$ is $\mathcal{F}_{k-1}$-measurable and the uniform random vector $U_k$ is independent of the $\sigma$-field $\mathcal{F}_{k-1}$ for all $k\in M(t)$.
Hence, we aim to construct stopping rules for the algorithm:
\begin{align}\label{algorithm}
\mu(t)&:=\frac{1}{|M(t)|} \sum_{k\in M(t)} \left(\Psi(U_k)+\left\langle{\bm \theta}(t-1), U_k-\mathbbm{1}_d/2\right\rangle\right), \\ 
v_2^2(t)&:=\frac{1}{|M(t)|} \sum_{k \in M(t)} \left(\Psi(U_k)+\left\langle{\bm \theta}(t-1), U_k-\mathbbm{1}_d/2\right\rangle\right)^2-\mu^2(t),\nonumber
%{\bm \theta}(t)=-\frac{12}{|M(t)|} \sum_{k\in M(t)} \Psi(U_k)(U_k-\mathbbm{1}_d/2),\quad t\in \mathbb{N},
\end{align}
along with \eqref{thetat CV}, each of which 
%${\bm \theta}(0)=0^d$ and the empirical variance
%\begin{equation}\label{adaptive CV empirical variance}
% v_2^2(t)=\frac{1}{|M(t)|} \sum_{k \in M(t)} \left(\Psi(U_k)+\left\langle{\bm \theta}(t-1), U_k-\mathbbm{1}_d/2\right\rangle\right)^2-\mu^2(t),\quad t\in \mathbb{N}.
%\end{equation}
%Note that this algorithm 
only requires elementary operations.

For illustrative purposes, we examine a one-dimensional integration problem with integrand $\Psi(u)=u^2/2$ on $(0,1)$, which has theoretical mean $\mu=1/6$. Our analysis shows the crude Monte Carlo estimator has variance $V(0)=2.22\times 10^{-2}$, while the theoretically minimized variance with optimal control variate is $V(\theta^*)=1.39\times 10^{-3}$ at $\theta^*=-1/2$.

\begin{comment}
\textcolor{red}{(I want you to cook figures in the way of \cite[Figures 1 and 3]{saaccelmcBAVR}, that is, without legends within, but with descriptions in the caption, and with marks, such as -o-, -$\diamond$-, etc.)}
\end{comment}
%We have that
%\[
%\mu:=\int_{(0,1)} \Psi(u) d u=\int_{(0,1)} \left(\Psi(u)+\left\langle\xi, u-\frac{1}{2}\right\rangle\right) d u = \frac{1}{6}.
%\]
%And, the optimizer is
%\[
%{\theta}^{*} = -12\int_{0}^1 \frac{u^2}{2}\left(u-\frac{1}{2}\right)=-\frac{1}{2}.
%\]
%Moreover, we can compute the original and reduced variance in the following way
%\[
%V\left(0\right) = \int_{0}^{1} \frac{u^4}{4} du-\frac{1}{36}=2.22\times 10^{-2}, \quad V\left(-\frac{1}{2}\right) = \int_{0}^{1} \left(\frac{u^2}{2}-\left\langle\frac{1}{2}, u-\frac{1}{2}\right\rangle\right)^2 du-\frac{1}{36}=1.39\times 10^{-3}.
%\]

\begin{figure}[H]
    \centering
    \includegraphics[width=\textwidth]{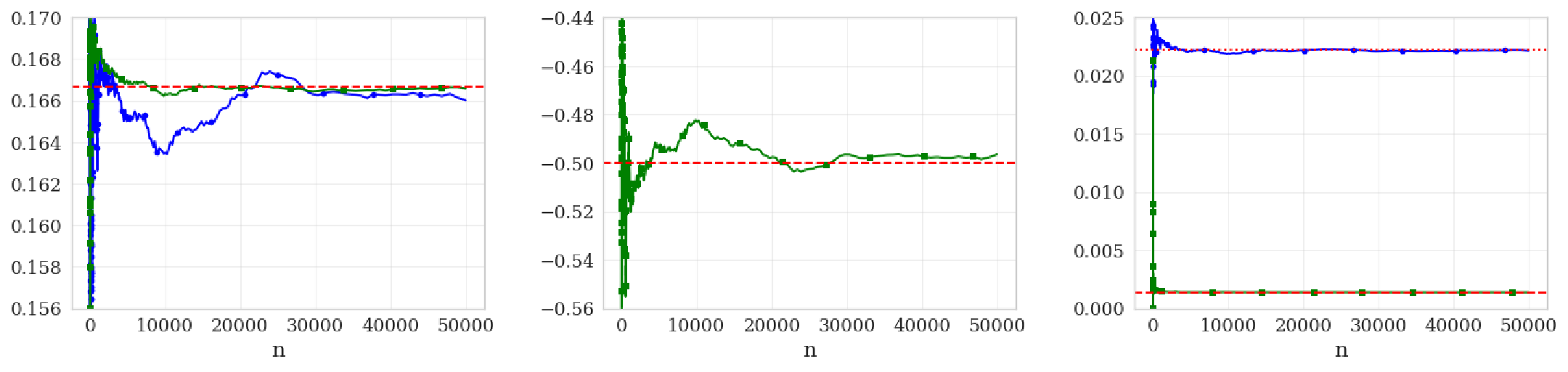} % Replace with actual path to your image
    \caption{%Sequential performance of the adaptive control variate method.
    %From left to right, 
    We plot empirical means $\mu(n)$, showing convergence behavior of both crude Monte Carlo (blue) and control variates (green) estimators towards the theoretical mean $\mu=1/6$ (dotted red) in the leftmost, an evolution of the adaptive control variates parameter $\theta(n)$ with its limiting value $\theta^*=-1/2$ (dotted red) in the middle, and empirical variances $v^2_2(n)$, with (green) and without (blue) adaptive control variates, with their respective limiting values $V(0)=2.22\times 10^{-2}$ and $V(\theta^*)=1.39\times 10^{-3}$
    %demonstrating the effect of variance reduction
    in the rightmost.
    %\textcolor{blue}{(No dotted red for $V(0)=2.22\times 10^{-2}$?)}
    %Each statistic is computed using cumulative samples up to sample size $n$ to show the full convergence behavior.
}
    \label{fig:samplepaths}
\end{figure}

In Figure \ref{fig:samplepaths} above, we present numerical results based upon the cumulative sample size $n$ to offer a more continuous view of the convergence behavior, as opposed to the batch index $t$ used in our theoretical development.
The plots provide several important insights, as follows.
The control variate estimator (green) demonstrates notably faster convergence to the theoretical mean $\mu=1/6$ (dotted red), compared to the crude Monte Carlo estimator (blue), with visibly reduced fluctuations from early stages.
The control variate parameter $\theta(n)$ shows initial volatility but steadily converges to the theoretical optimum $\theta^*=-0.5$ (dotted red).
This convergence corresponds directly to the improved performance in mean estimation and variance reduction.
The empirical variance plot clearly demonstrates the effectiveness of adaptive control variates, with the control variate estimator (green) achieving and maintaining significantly lower variance compared to the crude Monte Carlo counterpart (blue).

In Figure \ref{fig:complexity_reliability_example2} and Table \ref{tab:stopping_rule_metrics_acv} below, we present comprehensive performance metrics across different precision and error probability 
parameters. 
The results demonstrate that our stopping rules maintain robust performance even in the presence of adaptive 
parameter estimation ($\theta(t)$) and the inherent correlation structure between successive batches. 
We again set $a(t) = 1/t$ as the inflation factor in the stopping rule \eqref{stopping based on sigma(t)}.
%, and let \textcolor{red}{$l=5$} in \eqref{reliability complexity}.
The reliability consistently exceeds 1 (ranging from 1.0010 to 1.1111) despite the additional complexity introduced by the concurrent parameter optimization and variance reduction processes. The complexity metrics reveal relatively efficient computation 
compared to the ARCH(1) case, with a mean complexity of approximately 63,881 samples - about 25 times lower than the 
ARCH(1) scenario. 
This improved efficiency can be attributed to the variance reduction effect of the adaptive control 
variates, which progressively optimizes $\theta(t)$ to minimize the estimated variance. 
The complexity scaling remains predictable across the parameter space, ranging from about 252 samples for larger $\varepsilon$ and $\delta$ values to around 542,264 samples in the most demanding case ($\varepsilon = \delta = 10^{-3}$), demonstrating that the 
stopping rule effectively balances computational cost with precision requirements even as the control variate parameters 
are being adaptively optimized.
\begin{figure}[H]
    \centering
    \includegraphics[width=\textwidth]{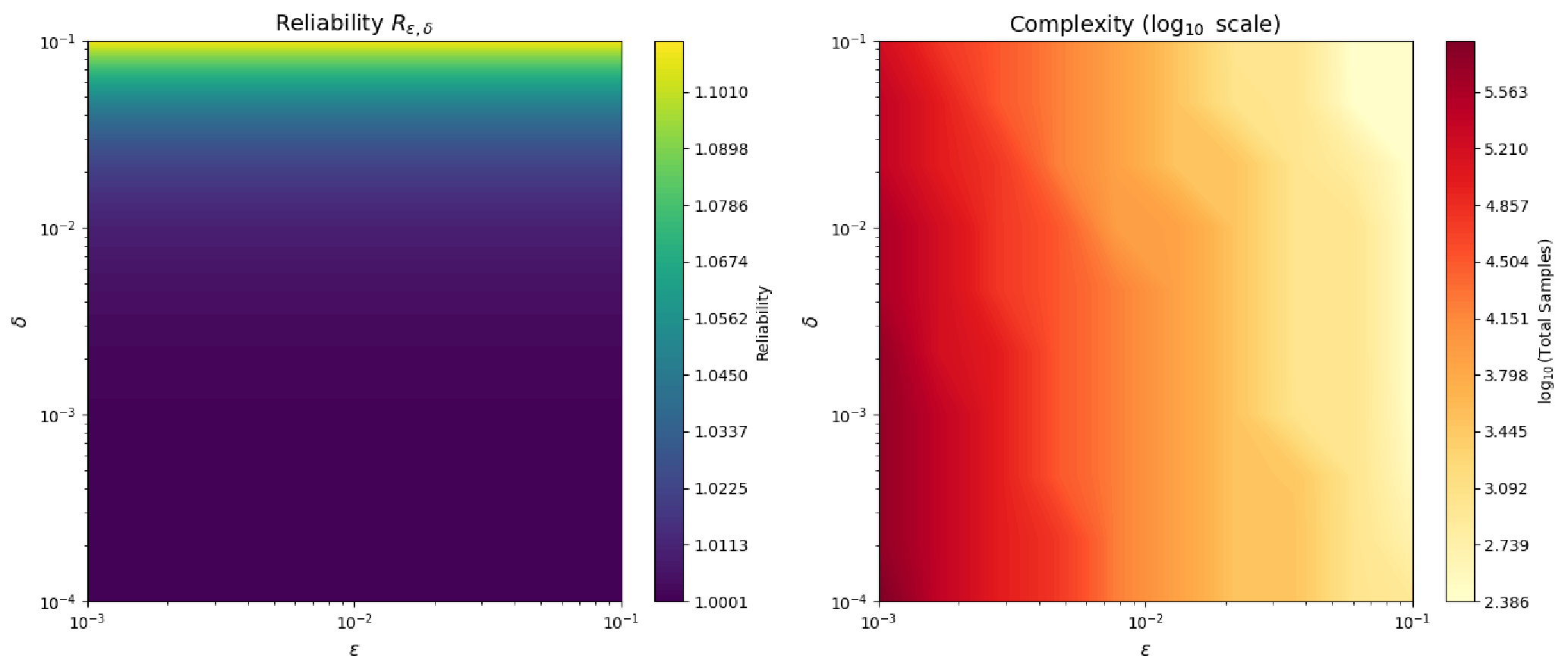} % Replace with actual path to your image
    \caption{The left plot represents the reliability $R_{\epsilon,\delta}$ and the right plot shows the complexity ${\rm CM}_{\epsilon,\delta}$ on a log10 scale.}
    \label{fig:complexity_reliability_example2}
\end{figure}

\begin{table}[H]
\centering
\begin{tabular}{lccc}
\hline
 & Mean & Min & Max \\
\hline
Reliability ($R_{\epsilon,\delta}$) & 1.0265 & 1.0010 & 1.1111 \\
Complexity ($\rm CM_{\epsilon,\delta}$) & 63880.9398 & 252.0000 & 542264.0200 \\
\hline
\end{tabular}
\caption{Performance metrics}
\label{tab:stopping_rule_metrics_acv}
\end{table}

\begin{remark}{\rm \label{remark ML}
It is worth noting that the proposed stopping rules for adaptive control variates have an interesting connection to averaged stochastic gradient descent (ASGD) \cite{Bottou2018}. While conventional SGD implementations often rely on validation performance for early stopping \cite{prechelt2012}, monitoring averaged iterates, which helps reduce the effect of the noise term, has proven beneficial in practice. This suggests potential value in developing statistically driven stopping criteria that leverage averaging, as in our framework, particularly for validation-free scenarios  \cite{duvenaud2016}.
\qed}\end{remark}

\section{Concluding remarks}
\label{section concluding remarks}

We have developed an easy-to-implement sequential stopping rule for Monte Carlo methods with a primal view towards the mean estimation of a non-iid sequence of martingale difference type in the non-asymptotic regime.
In contrast to standard iid Monte Carlo methods, where the classical central limit theorem suffices, the present aim relies largely on the martingale central limit theorem, which is, however, impractical in its original for most examples.
We have addressed this difficulty by carefully imposing a range of step-by-step technical conditions so that the corresponding stopping rules remain feasible.
We have provided numerical results, illustrating the applicability and effectiveness of the proposed stopping rules.

We conclude this work by highlighting two promising directions for future research.
The first involves improving the quality of the proposed stopping rules.
Unlike stopping rules based on the classical central limit theorem (such as \cite{BHST, FLYA2012}), a key challenge in enhancing the quality of stopping rules here arises from the lack of accurate finite-time approximations for the limiting law in the context of the martingale central limit theorem.
Specifically, obtaining Berry-Esseen-type bounds based on higher-order moments, and moreover presenting them in an implementable form, remains out of reach.
Although there is extensive literature on martingale central limit theorems and their rates of convergence - typically expressed as uniform bounds (such as \cite{10.1214/21-AIHP1182, FAN20191028, 10.3150/07-BEJ6116}) and, in some cases, as non-uniform bounds (such as \cite{EK1, pcc81, JOOS1991297}) - these results generally involve unspecifiable constants, limiting their practical applicability to the development of stopping rules, with an exception \cite{sydney}, which, to the best of our knowledge, provides the only implementable result of the uniform bound of Berry-Esseen type.

The second direction focuses on applying the proposed stopping rules across various fields.
In this study, we did not perform a comprehensive numerical analysis across a wide range of practical problem settings.
Given the advent of machine learning methods as briefly described in the introductory section and Remark \ref{remark ML}, it would be especially beneficial to adapt and apply the proposed stopping rules to these relevant contexts in a tailored manner.
Those would be valuable future research directions, deserving of their own separate investigation.

\section{Proofs}\label{section proofs}

To maintain the flow of the paper, we collect proofs here.
We try to keep the paper fairly self-contained, whereas we skip nonessential details of a somewhat routine nature in some instances to avoid overloading the Appendix.
In line with Figure \ref{figure what to output}, recall that $\mathbb{E}^{\star}_{k-1}$ is the expectation taken under the probability measure $\mathbb{P}^{\star}_{k-1}$ on the $\sigma$-field $\mathcal{F}^{\star}_{k-1}$ that makes the branched sequence $(X_1,\cdots,X_{m(\tau-1)},X^{\star}_{m(\tau-1)+1},\cdots,X^{\star}_{k-1})$ measurable for all $k\in M(\tau)$ and leaves the original stopping batch $(X_{m(\tau-1)+1},\cdots,X_{m(\tau)})$ out of its measurability.

\begin{proof}[Proof of Theorems \ref{theorem law of large numbers} and \ref{theorem unbiasedness}]
Theorem \ref{theorem law of large numbers} is elementary by the tower rule and the standing assumption $\mathbb{E}_{k-1}[X_k]=\mu$ (Assumption \ref{standing assumption 1}) for the unbiasedness, as well as due to, for instance, \cite[VII.9]{W1971} for the strong law of large numbers.
For Theorem \ref{theorem unbiasedness}, the first result can be derived as
\begin{align}\label{unbiasedness}
 \mathbb{E}_{m(\tau)}[\mu_{\star}(\tau)]&= \frac{1}{|M(\tau)|}\sum_{k\in M(\tau)}\mathbb{E}_{m(\tau)}\left[X_k^{\star}\right]=\frac{1}{|M(\tau)|}\sum_{k\in M(\tau)} \mathbb{E}_{m(\tau-1)}\left[X_k^{\star}\right]\nonumber\\
 &=\frac{1}{|M(\tau)|}\sum_{k\in M(\tau)} \mathbb{E}_{m(\tau-1)}\left[\mathbb{E}_{k-1}^{\star}\left[X_k^{\star}\right]\right]
 =\frac{1}{|M(\tau)|}\sum_{k\in M(\tau)} \mathbb{E}_{m(\tau-1)}[\mu], %=\mu,\quad \mathbb{E}_0[\mu(t)]=\mathbb{E}_0\left[\mathbb{E}_{m(t-1)}[\mu(t)]\right]=\mu,
\end{align}
where we have applied the $\mathcal{F}_{m(\tau)}$-measurability of the stopping batch $\tau$ for the first equality, the regeneration procedure (illustrated in Figure \ref{figure what to output} and described alongside the statement of Theorem \ref{theorem unbiasedness}) for the second equality, and the tower rule with $\mathbb{E}^{\star}_{k-1}[X_k^{\star}]=\mathbb{E}_{k-1}[X_k]=\mu$ (Assumption \ref{standing assumption 1}) for the third and fourth equalities, due to $\mathcal{F}_{m(\tau-1)}\subseteq \mathcal{F}^{\star}_{k-1}$ for all $k\in M(\tau)=\{m(\tau-1)+1,\cdots,m(\tau)\}$.
The second result is then straightforward due to $\mathcal{F}_0\subseteq \mathcal{F}_{m(\tau)}$.
\end{proof}

\begin{proof}[Proof of Theorem \ref{theorem base MCLT}]
A standard martingale central limit theorem (for instance, \cite[Theorem 2.5 (a)]{ish} with $\xi_{t,k}=(X_k-\mu)/\sqrt{|M(t)|v^2_0(t)}$) reads in the present framework as follows.
If $v_1^2(t)/v^2_0(t)\stackrel{\mathbb{P}_0}{\to} 1$ and   
$(|M(t)|v^2_0(t))^{-1} \sum_{k\in M(t)} \mathbb{E}_{k-1}[|X_k-\mu|^2 \mathbbm{1}(|X_k-\mu|>\epsilon \sqrt{|M(t)|}v_0(t)) ]\stackrel{\mathbb{P}_0}{\to}0$ for every $\epsilon>0$,
then it holds that $\mathbb{P}_0(\sqrt{|M(t)|}(\mu(t)-\mu)/v_0(t)\le x) \to \Phi(x)$ for all $x\in \mathbb{R}$, as $t\to +\infty$.
This Lindeberg condition can be simplified as the one in Theorem \ref{theorem very base MCLT}, thanks to $\inf_{t\in\mathbb{N}}v^2_0(t)>0$ and $\sup_{t\in \mathbb{N}}v^2_0(t)<+\infty$ (Assumption \ref{standing assumption 2}).
The desired weak convergence \eqref{first MCLT} holds true by the Slutsky theorem due to the condition  $v_1^2(t)/v^2_0(t)\stackrel{\mathbb{P}_0}{\to} 1$.
%For (ii), we refer the reader to \cite[Theorem 8.2.8]{durrett_2019}.
\end{proof}

% \textcolor{blue}{Wish to find a mild condition for $v_1^2(t)/v^2_0(t)\stackrel{\mathbb{P}_0}{\to} 1$ to hold true.
% \[
%  \mathbb{P}_0\left(\left|v_1^2(t)/v^2_0(t)-1\right|>\epsilon\right)\le \frac{1}{\epsilon\inf_{t\in \mathbb{N}}v^2_0(t)}\mathbb{E}_0\left[\left|v_1^2(t)-v^2_0(t)\right|\right]\le \frac{1}{\epsilon\inf_{t\in \mathbb{N}}v^2_0(t)}\frac{1}{|M(t)|}\sum_{k\in M(t)}\mathbb{E}_0\left[\left|\mathbb{E}_{k-1}[X_k^2]-\mathbb{E}_0[X_k^2]\right|\right]
% \]}
% \textcolor{red}{\begin{remark}{\rm
% We note that the condition $v_1^2(t)/v^2_0(t)\stackrel{\mathbb{P}_0}{\to} 1$ of Theorem \ref{theorem base MCLT} implies $v_1^2(t)-v^2_0(t)\to 0$ in $L^1(\Omega)$ under $\limsup_{t\to +\infty}v^2_0(t)<+\infty$ (Assumption \ref{standing assumption 2}), since 
% \[
%  \mathbb{E}_0\left[\left|v_1^2(t)-v^2_0(t)\right|\right]\le \left(\sup_{t\in \mathbb{N}}v^2_0(t)\right)\mathbb{E}_0\left[\left|v_1^2(t)/v^2_0(t)-1\right|\right]\to 0,
% \]
% where the convergence holds true due to the condition $v_1^2(t)/v^2_0(t)\stackrel{\mathbb{P}_0}{\to} 1$, together with non-negativity $v_1^2(t)/v^2_0(t)\ge 0$ and unbiasedness $\mathbb{E}_0[v_1^2(t)/v^2_0(t)]=\mathbb{E}_0[v_1^2(t)]/v^2_0(t)=1$.
% \qed}\end{remark}}

Next, we review a standard result on the strong law of large numbers for a sequence of asymptotically pairwise uncorrelated random variables.
This result has already been employed in Section \ref{section ARCH(1)}, as well as will be used in the proof of Theorem \ref{theorem empirical MCLT}.

\begin{lemma}\label{lemma SLLN uncorrelated}
Let $\{Y_k\}_{k\in\mathbb{N}}$ be a sequence of random variables, satisfying the following conditions: $\lim_{k\to +\infty}\mathbb{E}_0[Y_k](=:c)$ exists, $|\mathbb{E}_0[(Y_k-\mathbb{E}_0[Y_k])(Y_m-\mathbb{E}_0[Y_m])]|=\mathcal{O}(m^{-2})$ for all $k\in\mathbb{N}$, and $\sup_{k\in\mathbb{N}}{\rm Var}_0(Y_k)<+\infty$.
Then, it holds $\mathbb{P}_0$-$a.s.$ that $m^{-1}\sum_{k=1}^m \mathbb{E}_0[Y_k]\to c$, as $m\to +\infty$.
\end{lemma}

\begin{proof}[Proof of Lemma \ref{lemma SLLN uncorrelated}]
%https://math.stackexchange.com/questions/1076980/strong-law-of-large-numbers-for-square-integrable-and-uncorrelated-random-variab?rq=1
It holds by the Chebyshev inequality that for every $\epsilon>0$,
\begin{multline*}
\sum_{m=1}^{+\infty}\mathbb{P}_0\left(\left|\frac{1}{m^2} \sum_{k=1}^{m^2}\left(Y_k-\mathbb{E}_0\left[Y_k\right]\right)\right| \geq \epsilon\right) =\sum_{m=1}^{+\infty}\frac{1}{\epsilon^2 m^4}\sum_{k=1}^{m^2} {\rm Var}_0(Y_k)\\
+\sum_{m=1}^{+\infty} \frac{1}{\epsilon^2 m^4} \sum_{k_1\ne k_2}^{m^2}\mathbb{E}_0\left[\left(Y_{k_1}-\mathbb{E}_0\left[Y_{k_1}\right]\right)\left(Y_{k_2}-\mathbb{E}_0\left[Y_{k_2}\right]\right)\right] <+\infty,
\end{multline*}
where we have applied asymptotic pairwise uncorrelation.
Hence, the Borel-Cantelli lemma yields
$\lim_{m \to +\infty} m^{-2} \sum_{k=1}^{m^2}(Y_k-\mathbb{E}_0[Y_k])=0,$ $\mathbb{P}_0$-$a.s.$
In order to de-square the number of summands $m^2$, let $m\in\mathbb{N}$ and let $n_m$ be the smallest index such that $m\in\{n_m^2, \cdots, (n_m+1)^2\}$.
Then, decompose the sum into two parts: 
\[
\frac{1}{m} \sum_{k=1}^{m}\left(Y_k-\mathbb{E}_0\left[Y_k\right]\right)=\frac{1}{m} \sum_{k=1}^{(n_{m}+1)^2}\left(Y_k-\mathbb{E}_0\left[Y_k\right]\right)-\frac{1}{m} \sum_{k=m+1}^{(n_{m}+1)^2}\left(Y_k-\mathbb{E}_0\left[Y_k\right]\right),
\]
each of which vanishes $\mathbb{P}_0$-$a.s.$ as $m\to +\infty$, because $(n_m+1)^2\sim m$ as $m\to +\infty$.
%For the second term, observe that, by the Chebyshev inequality, for every $\epsilon>0$,
%\[
%\sum_{m=1}^{+\infty} \mathbb{P}_0\left(\left|\frac{1}{m} \sum_{k=m+1}^{(n_m+1)^2}\left(X_k-\mathbb{E}_0\left[X_k\right]\right)\right| \geq \epsilon\right) \leq \frac{\sigma^2_0}{\epsilon^2} \sum_{m=1}^{+\infty} \frac{(n_m+1)^2-n_m^2}{n_m^2} \leq \frac{\sigma^2_0}{\epsilon^2} \sum_{m=1}^{+\infty}\left(\frac{(n_m+1)^2}{n_m^2}-1\right)^2  <+\infty,
%\]
%where the last inequality holds because of the convergence of geometric series.
%This concludes the proof.
\end{proof}

\begin{proof}[Proof of Theorem \ref{theorem empirical MCLT}]
Thanks to $\inf_{t\in\mathbb{N}}v^2_0(t)>0$ (Assumption \ref{standing assumption 2}) and $v_1^2(t)/v^2_0(t)\stackrel{\mathbb{P}_0}{\to}1$ (Theorem \ref{theorem base MCLT}), it suffices to show that $v_2^2(t)-v_1^2(t)\stackrel{\mathbb{P}_0}{\to} 0$.
Let $\{\xi_k\}_{k\in\mathbb{N}}$ and $\{\eta_k\}_{k\in\mathbb{N}}$ be sequences of random variables defined by $\xi_k:=X_k-\mu$ and $\eta_k:=X_k^2-\mathbb{E}_{k-1}[X_k^2]$, with which we obtain
\[
v_2^2(t)-v_1^2(t)=\frac{1}{|M(t)|}v_2^2(t)+\frac{2\mu}{|M(t)|}\sum_{k \in M(t)} \xi_k + \frac{1}{|M(t)|}\sum_{k \in M(t)}\eta_k+\mu^2-\mu^2(t),\quad t\in \mathbb{N}.
\]
%Recalling $\mathbb{E}_{k-1}[X_k]=\mu$ for all $k\in\mathbb{N}$, it is clear that each sequence forms a martingale difference sequence with respect to the filtration $(\mathcal{F}_k)_{k\in\mathbb{N}_0}$.
It holds by Lemma \ref{lemma SLLN uncorrelated} that the first two terms tend to vanish.
The third term tends to zero due to the second condition given in Theorem \ref{theorem empirical MCLT}. %\eqref{quaratic condition}.
Finally, it holds by Theorem \ref{theorem law of large numbers} and the continuous mapping theorem that the last two terms tend to cancel as $t\to +\infty$.
% \textcolor{red}{It thus remains to show the second term also tends to zero.
% Note that $\{\eta_k\}_{k\in\mathbb{N}}$ forms a martingale difference sequence with respect to the filtration $(\mathcal{F}_k)_{k\in\mathbb{N}_0}$.
% It thus also holds true that its elements are asymptotically pairwise uncorrelated.
% Finally, recalling $M(t)=\{m(t-1)+1,\cdots,m(t)\}$, observe that $\limsup_{t\to +\infty}\sum_{k\in M(t)}(k-m(t-1))^{-2} \mathbb{E}_{k-1}[\eta_k^2]$
% %\[
% %\sum_{k\in M(t)}\frac{1}{(k-m(t-1))^2} \mathbb{E}_{k-1}\left[\eta_k^2 \right]= \sum_{k\in M(t)} \frac{1}{(k-m(t-1))^2}\mathbb{E}_{k-1}\left[\xi_k^{4} \right]- \sum_{k\in M(t)} \frac{1}{(k-m(t-1))^2}\left(\mathbb{E}_{k-1}\left[\xi_k^2 \right]\right)^2,\quad t\in \mathbb{N},
% %\]
% %whose supremum over $t\in\mathbb{N}$ 
% is $\mathbb{P}_0$-$a.s.$ finite under the condition \eqref{quaratic condition}. 
% With the aid of \cite[VII.9]{W1971}, the proof is complete.} 
\end{proof}

Finally, we validate the asymptotic results of the three stopping rules \eqref{def of tau0}, \eqref{stopping based on general v(t)}, and \eqref{stopping based on sigma(t)}.
%and \eqref{empirical criterion based on berry-esseen}.
We deal with the first one \eqref{stopping based on general v(t)} separately from the remaining three, since the first one is deterministic whereas the other two are probabilistic along with the strictly positive inflation $a(t)$.
%To ease the notation, we write $\Phi^{-2}(x):=(\Phi^{-1}(x))^2$ for all $x\in \mathbb{R}$.

\begin{proof}[Proof of Proposition \ref{proposition asymptotic validity of tau0}]
The first claim is trivial, since every element of $\{X_k\}_{k\in\mathbb{N}}$ is not degenerate due to $\inf_{t\in\mathbb{N}}v^2_0(t)>0$ (Assumption \ref{standing assumption 2}).
Next, observe that $\mathbb{P}_0(|\mu(\tau_{\epsilon,\delta}^0-1)-\mu|>\epsilon)>\delta$ and $\mathbb{P}_0(|\mu(\tau_{\epsilon,\delta}^0)-\mu|>\epsilon)\le \delta$, by the definition \eqref{def of tau0}.
Along with $\tau_{\epsilon,\delta}^0\to +\infty$ as $\epsilon\to 0+$ and the martingale central limit theorem \eqref{very first MCLT}, those can be approximated as follows: 
\begin{gather*}
    \delta\lesssim 2(1-\Phi(\epsilon\sqrt{|M(\tau_{\epsilon,\delta}^0-1)|}/v(\tau_{\epsilon,\delta}^0-1))),\\
    2(1-\Phi(\epsilon\sqrt{|M(\tau_{\epsilon,\delta}^0)|}/v(\tau_{\epsilon,\delta}^0)))\lesssim \delta,
\end{gather*}
for sufficiently small $\epsilon$.
Hence, we get
\begin{multline}\label{sandwich0}
 (\Phi^{-1}(1-\delta/2))^2\lesssim \frac{\epsilon^2 |M(\tau^0_{\epsilon,\delta})|}{v^2(\tau^0_{\epsilon,\delta})}
 =\frac{\epsilon^2 |M(\tau^0_{\epsilon,\delta}-1)|}{v^2(\tau^0_{\epsilon,\delta}-1)}\frac{|M(\tau^0_{\epsilon,\delta})|}{|M(\tau^0_{\epsilon,\delta}-1)|}\frac{v^2(\tau^0_{\epsilon,\delta}-1)}{v^2(\tau^0_{\epsilon,\delta})}\\
 \lesssim \Phi^{-2}(1-\delta/2)\frac{|M(\tau^0_{\epsilon,\delta})|}{|M(\tau^0_{\epsilon,\delta}-1)|}\frac{v^2(\tau^0_{\epsilon,\delta}-1)}{v^2(\tau^0_{\epsilon,\delta})}\to (\Phi^{-1}(1-\delta/2))^2,
\end{multline}
as $\epsilon\to 0+$, due to $v^2(t+1)/|M(t+1)|\sim v^2_0(t)/|M(t)|$ as $t\to +\infty$ (Assumption \ref{standing assumption 2}).
This yields Part (ii).
Finally, for Part (iii), observe that
\begin{align*}
\epsilon^{1/q} \tau_{\epsilon,\delta}^0
&=\epsilon^{1/q}\min\left\{t\in\mathbb{N}:\,\mathbb{P}_0\left(\sqrt{|M(t)|}\frac{|\mu(t)-\mu|}{v_0(t)}> \sqrt{|M(t)|}\frac{\epsilon}{v_0(t)}\right)\le \delta\right\}\\
%&\sim \argmin_{t\in (0,+\infty)}\left\{\mathbb{P}_0\left(\sqrt{|M(\lfloor \epsilon^{-1/q}t\rfloor)|}\frac{|\mu(\lfloor \epsilon^{-1/q}t\rfloor)-\mu|}{v(\lfloor \epsilon^{-1/q}t\rfloor)}> \sqrt{|M(\lfloor \epsilon^{-1/q}t\rfloor)|}\frac{\epsilon}{v(\lfloor \epsilon^{-1/q}t\rfloor)}\right)\le \delta\right\}\\
& \sim \inf\left\{t\in (0,+\infty):\,2\left(1-\Phi\left(\epsilon \sqrt{|M(\lfloor \epsilon^{-1/q}t\rfloor)|}/v(\lfloor \epsilon^{-1/q}t\rfloor)\right)\right)\le \delta\right\},
\end{align*}
which yields the first result as $\epsilon\to 0+$, due to the martingale central limit theorem \eqref{very first MCLT}, as well as the convergence $\epsilon\sqrt{|M(\lfloor \epsilon^{-1/q}t\rfloor)|}/v(\lfloor \epsilon^{-1/q}t\rfloor) \to t^q/\sqrt{c}$ as $\epsilon\to 0+$.
Finally, observe that
\begin{align}
\epsilon^{-1}(\mu(\tau_{\epsilon,\delta}^0)-\mu)&\sim \frac{(\tau_{\epsilon,\delta}^0)^q}{\sqrt{c}\Phi^{-1}(1-\delta/2)}(\mu(\tau_{\epsilon,\delta}^0)-\mu)\nonumber\\
&\sim \frac{1}{\Phi^{-1}(1-\delta/2)} \sqrt{|M(\tau_{\epsilon,\delta}^0)|}\frac{\mu(\tau_{\epsilon,\delta}^0)-\mu}{v(\tau_{\epsilon,\delta}^0)},\label{scaled second order for tau0}
\end{align}
which yields the desired weak convergence due to the martingale central limit theorem \eqref{very first MCLT} and $\tau_{\epsilon,\delta}^0\to +\infty$.
\end{proof}

\begin{proof}[Proof of Proposition \ref{proposition asymptotic validity 1}]
%and \ref{proposition asymptotic validity 3}]
For Proposition \ref{proposition asymptotic validity 1}, the expression \eqref{stopping based on general v(t)} can be rewritten as $\tau_{\epsilon,\delta}=\min\{t\in \mathbb{N}:\,\epsilon\sqrt{|M(t)|}/(v(t)+a(t))\ge \Phi^{-1}(1-\delta/2)\}$, which yields $\tau_{\epsilon,\delta}\stackrel{a.s.}{\to}+\infty$ as $\epsilon\to 0+$ or $\delta\to 0+$, due to $|M(t)|\to +\infty$ (Assumption \ref{standing assumption M}), $v^2(t)/v^2_0(t)\stackrel{a.s.}{\to}1$ (Theorem \ref{theorem base MCLT}) and $\inf_{t\in\mathbb{N}}v^2_0(t)>0$ (Assumption \ref{standing assumption 2}).
With the aid of this expression, we get $\epsilon \sqrt{|M(\tau_{\epsilon,\delta})|}/(v(\tau_{\epsilon,\delta})+a(\tau_{\epsilon,\delta}))\ge \Phi^{-1}(1-\delta/2)$ at the time $m(\tau_{\epsilon,\delta})$, and $\epsilon \sqrt{|M(\tau_{\epsilon,\delta}-1)|}/(v(\tau_{\epsilon,\delta}-1)+a(\tau_{\epsilon,\delta}-1))<\Phi^{-1}(1-\delta/2)$ at the time $m(\tau_{\epsilon,\delta}-1)$, which together give, in a similar manner to \eqref{sandwich0}, 
\begin{align}\nonumber
 &\Phi^{-1}(1-\delta/2) \le \frac{\epsilon \sqrt{|M(\tau_{\epsilon,\delta})|}}{v(\tau_{\epsilon,\delta})+a(\tau_{\epsilon,\delta})}\nonumber\\
 &\, =\frac{\epsilon \sqrt{|M(\tau_{\epsilon,\delta}-1)|}}{v(\tau_{\epsilon,\delta}-1)+a(\tau_{\epsilon,\delta}-1)}\frac{\sqrt{|M(\tau_{\epsilon,\delta})|}}{\sqrt{|M(\tau_{\epsilon,\delta}-1)|}}\frac{v(\tau_{\epsilon,\delta}-1)+a(\tau_{\epsilon,\delta}-1)}{v(\tau_{\epsilon,\delta})+a(\tau_{\epsilon,\delta})}\nonumber\\
 &\, < \Phi^{-1}(1-\delta/2)\frac{\sqrt{|M(\tau_{\epsilon,\delta})|}}{\sqrt{|M(\tau_{\epsilon,\delta}-1)|}}\frac{v(\tau_{\epsilon,\delta}-1)+a(\tau_{\epsilon,\delta}-1)}{v(\tau_{\epsilon,\delta})+a(\tau_{\epsilon,\delta})}\stackrel{a.s.}{\to} \Phi^{-1}(1-\delta/2),\label{sandwich}
\end{align}
as $\epsilon \to 0+$, due to the first result $\tau_{\epsilon,\delta}\stackrel{a.s.}{\to}+\infty$, as well as  $v^2(t)/v^2_0(t)\stackrel{a.s.}{\to}1$ (Theorem \ref{theorem base MCLT}), $v^2(t+1)/|M(t+1)|\sim v^2_0(t)/|M(t)|$ and $\inf_{t\in\mathbb{N}}v^2_0(t)>0$ (Assumption \ref{standing assumption 2}), and $\lim_{t\to +\infty}a(t)=0$.
Hence, Part (ii) holds by the continuous mapping theorem since $\Phi$ is continuous on $\mathbb{R}$.  
With the additional condition, the first result of Part (iii) can be obtained by reformulating the equation \eqref{sandwich}.
The second result of Part (iii) holds by the asymptotic equivalence \eqref{scaled second order for tau0} with the regenerated empirical mean $\mu_{\star}$ and the stopping rule $\tau_{\epsilon,\delta}^a$,  
%\[
%\epsilon^{-1}(\mu_{\star}(\tau_{\epsilon,\delta}^a)-\mu)\sim \frac{(\tau_{\epsilon,\delta}^a)^q}{\sqrt{c}\Phi^{-1}(1-\delta/2)}(\mu_{\star}(\tau_{\epsilon,\delta}^a)-\mu)\sim \frac{1}{\Phi^{-1}(1-\delta/2)} \sqrt{|M(\tau_{\epsilon,\delta}^a)|}\frac{\mu_{\star}(\tau_{\epsilon,\delta}^a)-\mu}{v(\tau_{\epsilon,\delta}^a)},
%\]
which yields the desired weak convergence due to the martingale central limit theorem \eqref{first MCLT} with the aid of $v^2(t)/v^2_0(t)\stackrel{a.s.}{\to}1$ (Theorem \ref{theorem base MCLT}) and $\tau_{\epsilon,\delta}^a\stackrel{a.s.}{\to}+\infty$ (Proposition \ref{proposition asymptotic validity 1} (i)), where the two asymptotic equivalences above hold almost surely, respectively, due to the first result and the additional condition of Part (iii).
\end{proof}

\small
\bibliographystyle{abbrv}
\bibliography{ref}

\newpage
\setcounter{page}{1}
\begin{center}
	{\bf \Large Supplementary materials: Stopping rules for Monte Carlo methods of martingale difference type}
    %\\
 
% {\bf \Large Stopping rules for Monte Carlo methods of martingale difference type}
\end{center}
\vspace{1em}

We provide supplementary materials for ``Stopping rules for Monte Carlo methods of martingale difference type,'' which we skipped in the main text to maintain the flow of the paper.

\appendix

\section{Verification of technical assumptions and conditions}\label{appendix technical justifications}

We here collect verifications of the standing assumptions (Assumptions \ref{standing assumption 1}, \ref{standing assumption M} and \ref{standing assumption 2}) and the technical conditions 
(Theorems \ref{theorem law of large numbers}, \ref{theorem very base MCLT}, \ref{theorem base MCLT} and  \ref{theorem empirical MCLT})
%\textcolor{red}{and \ref{theorem implementable berry esseen}}) 
imposed along the way on the two examples examined in Section \ref{section numerical illustations}.
%We do not discuss adaptive importance sampling (Section \ref{section adaptive importance sampling}) here as we have already dealt with its technical aspects there.

\subsection{ARCH(1)}
\label{appendix ARCH}

Recall that $X_k=(\beta+\alpha X_{k-1}^2)^{1/2}V_k$ for all $k\in \mathbb{N}$ with $X_0=0$ and $\{V_k\}_{k\in \mathbb{N}}$ is a sequence of iid random variables with the $t$-distribution with $\nu$ degrees of freedom for some $\nu\in\{5,6,\cdots\}$, scaled in such a way to satisfy $\mathbb{E}_0[V_1^2]=1$, that is, the common probability density function is given by
\begin{equation}\label{T pdf} 
 \frac{d}{dx}\mathbb{P}(V_1\le x)=\frac{\Gamma((\nu+1)/2)}{\sqrt{(\nu-2)\pi}\Gamma(\nu/2)}\left(1+\frac{x^2}{\nu-2}\right)^{-(\nu+1)/2},\quad x\in\mathbb{R}.
\end{equation}
In particular, we have $\mathbb{E}_0[|V_1|^4]=3(\nu-2)/(\nu-4)$.
It is trivial to derive that $\mathbb{E}_{k-1}[X_k]=(\beta+\alpha X_{k-1}^2)^{1/2}\mathbb{E}_{k-1}[V_k]=0$ and $\mathbb{E}_{k-1}[X_k^2]= \beta+\alpha X_{k-1}^2$, due to $X_{k-1}\in \mathcal{F}_{k-1}$, $V_k\perp \mathcal{F}_{k-1}$, and $\mathbb{E}_{k-1}[V_k^2]={\rm Var}_{k-1}(V_k)={\rm Var}_0(V_k)=1$ for all $k\in \mathbb{N}$.
Hence, we get $\mathbb{E}_0[X_k]=0$ (Assumption \ref{standing assumption 1} with $\mu=0$) and ${\rm Var}_0(X_k)=\mathbb{E}_0[X_k^2]=\beta +\alpha \mathbb{E}_0[X_{k-1}^2]=\cdots=\beta(1-\alpha^k)/(1-\alpha)>0$, due to $\alpha\in (0,1)$, $\beta \in (0,+\infty)$ and $X_0=0$.
Hence, the discrete-time stochastic process $\{X_k:\,k\in\mathbb{N}_0\}$ is asymptotically weakly stationary with mean $\mathbb{E}_0[X_k]=0$ and variance $\lim_{k\to +\infty}{\rm Var}_0(X_k)=\beta/(1-\alpha)$, due to $\alpha\in (0,1)$.
Hence, we get the convergence of the theoretical batch variance as
\[
v^2_0(t) =\frac{1}{|M(t)|} \sum_{k \in M(t)} {\rm Var}_0(X_k)=\frac{\beta}{1-\alpha}\left(1-\frac{\alpha^{m(t-1)+1}-\alpha^{m(t)}}{2}\right)\to \frac{\beta}{1-\alpha},
\]
as $t\to +\infty$ due to $m(t)\to +\infty$.
Hence, we get the assumptions $\inf_{t\in\mathbb{N}}v^2_0(t)>0$ and $\sup_{t\in \mathbb{N}}v^2_0(t)<+\infty$ (Assumption \ref{standing assumption 2}), and thus $\limsup_{t\to+\infty}\sum_{k\in M(t)}(k-m(t-1))^{-2}{\rm Var}_0(X_k)<+\infty$ (Theorem \ref{theorem law of large numbers}).
In addition to $|M(t)|\to +\infty$ (Assumption \ref{standing assumption M}), we need to set the batch sizes $\{M(t)\}_{k\in\mathbb{N}}$ such that $|M(t+1)|\sim |M(t)|$ to ensure $v^2(t+1)/M(t+1)\sim v^2_0(t)/M(t)$ (Assumption \ref{standing assumption 2}).

One can show that all the relevant conditions hold true if the constant $\alpha$ and the degree of freedom $\nu$ satisfy $3\alpha^2(\nu-2)/(\nu-4)<1$.
To verify the convergence of the conditional batch variance, observe first that for every $k\in \mathbb{N}$ and $m\in \{k+1,\cdots\}$,
\begin{align*}
 \mathbb{E}_0\left[\left(X_k^2-\mathbb{E}_0\left[X_k^2\right]\right)\left(X_m^2-\mathbb{E}_0\left[X_m^2\right]\right)\right]
 &=\mathbb{E}_0\left[\left(X_k^2-\mathbb{E}_0\left[X_k^2\right]\right)\mathbb{E}_k\left[X_m^2-\mathbb{E}_0\left[X_m^2\right]\right]\right]\\
 &=\mathbb{E}_0\left[\left(X_k^2-\mathbb{E}_0\left[X_k^2\right]\right)\left(\beta+\alpha \mathbb{E}_k \left[X_{m-1}^2\right]-\mathbb{E}_0\left[X_m^2\right]\right)\right]\\
 %&=\alpha\mathbb{E}_0\left[\left(X_k^2-\mathbb{E}_0\left[X_k^2\right]\right)\mathbb{E}_k \left[X_{m-1}^2\right]\right]\\
 %&=\alpha\mathbb{E}_0\left[\left(X_k^2-\mathbb{E}_0\left[X_k^2\right]\right)X_{m-1}^2\right]\\
 &=\cdots=\alpha\mathbb{E}_0\left[\left(X_k^2-\mathbb{E}_0\left[X_k^2\right]\right)\left(X_{m-1}^2-\mathbb{E}_0 \left[X_{m-1}^2\right]\right)\right]\\
 &=\cdots=\alpha^{m-k}{\rm Var}_0(X_k^2),
\end{align*}
which tends to zero a lot faster than $m^{-2}$, due to $|\alpha|<1$, if $\sup_{k\in\mathbb{N}}{\rm Var}_0(X_k^2)<+\infty$.
To verify $\sup_{k\in\mathbb{N}}{\rm Var}_0(X_k^2)<+\infty$, observe that for every $k\in\mathbb{N}$,
\begin{align*}
\mathbb{E}_0\left[|X_k|^4\right]&=\mathbb{E}_0\left[(\beta+\alpha X_{k-1}^2)^2\right]\mathbb{E}_0\left[|V_k|^4\right]\\
&=\frac{3(\nu-2)}{\nu-4}\left(\beta^2+2\beta\alpha \mathbb{E}_0\left[X_{k-1}^2\right]+\alpha^2 \mathbb{E}_0\left[|X_{k-1}|^4\right]\right)\\
&\sim \frac{3(\nu-2)}{\nu-4}\left(\beta^2\frac{1+\alpha}{1-\alpha}+\alpha^2 \mathbb{E}_0\left[|X_{k-1}|^4\right]\right).
\end{align*}
If $3\alpha^2(\nu-2)/(\nu-4)<1$, then we get
\begin{gather*}
 \mathbb{E}_0\left[|X_k|^4\right]\to \frac{3\beta^2(1+\alpha)(\nu-2)}{(1-\alpha)(\nu-4-3\alpha^2(\nu-2))},\\
 {\rm Var}_0\left(X_k^2\right)\to \frac{3\beta^2(1+\alpha)(\nu-2)}{(1-\alpha)(\nu-4-3\alpha^2(\nu-2))}-\frac{\beta^2}{(1-\alpha)^2},
\end{gather*}
as $k\to +\infty$.
Hence, by Lemma \ref{lemma SLLN uncorrelated}, we get $\lim_{t\to +\infty}|M(t)|^{-1}\sum_{k\in M(t)}X_{k-1}^2= \lim_{k\to +\infty}\mathbb{E}_0[X_k^2]=\beta/(1-\alpha)$.
%that is, $\lim_{t\to +\infty}v_1^2(t)=\beta/(1-\alpha)$, $\mathbb{P}_0$-$a.s.$
That is, with $\mu=0$, we get %he conditional batch variance is given by 
\begin{align*}
v_1^2(t) &= \frac{1}{|M(t)|} \sum_{k \in M(t)} \mathbb{E}_{k-1}\left[X_k^2\right]-\mu^2 = \frac{1}{|M(t)|}\sum_{k\in M(t)}\left(\beta+\alpha X_{k-1}^2\right)\\
&=\beta+\frac{\alpha}{|M(t)|}\sum_{k\in M(t)}X_{k-1}^2\to \frac{\beta}{1-\alpha},
\end{align*}
$\mathbb{P}_0$-$a.s.$, as $t\to +\infty$.
With $v^2_0(t)\to \beta/(1-\alpha)$, the condition $v_1^2(t)/v^2_0(t)\stackrel{\mathbb{P}_0}{\to} 1$ (Theorem \ref{theorem base MCLT}) holds true.
In addition, the convergence of the sequence $\{\mathbb{E}_0[|X_k|^4]\}_{k\in\mathbb{N}}$ above implies the condition $\limsup_{k\to +\infty}\mathbb{E}_0[|X_k|^4]<+\infty$ (Theorem \ref{theorem empirical MCLT}). 
Now, with $c_{\nu}$ being a constant depending on $\nu$ changing its values from line to line, due to 
\begin{align*}
    \mathbb{E}_{k-1}\left[V_k^2\mathbbm{1}(|X_k|>\epsilon \sqrt{|M(t)|})\right]
    &=c_{\nu}
    %\frac{\Gamma((n+1)/2)}{\sqrt{(n-2)\pi}\Gamma(n/2)}
    \int_{x^2>\frac{\epsilon^2 |M(t)|}{\beta+\alpha X_{k-1}^2}}x^2\left(1+\frac{x^2}{\nu-2}\right)^{-(\nu+1)/2}dx\\
&    \sim c_{\nu}
    %\frac{2\Gamma((n+1)/2)}{\sqrt{\pi}\Gamma(n/2)}(n-2)^{n/2-1}
    \left(\frac{\epsilon^2 |M(t)|}{\beta+\alpha X_{k-1}^2}\right)^{1-\nu/2},
\end{align*}
as $t\to +\infty$ with the aid of the probability density function \eqref{T pdf}, we get
\begin{align*}
 &\frac{1}{|M(t)|}\sum_{k\in M(t)}\mathbb{E}_{k-1}\left[X_k^2\mathbbm{1}(|X_k|>\epsilon \sqrt{|M(t)|})\right]\\
 &\quad =\frac{1}{|M(t)|}\sum_{k\in M(t)} \left(\beta+\alpha X_{k-1}^2\right)\mathbb{E}_{k-1}\left[V_k^2\mathbbm{1}(|X_k|>\epsilon \sqrt{|M(t)|})\right]\\
 &\quad \sim c_{\nu} \epsilon^{2-n}\sum_{k\in M(t)} \left(\frac{\beta+\alpha X_{k-1}^2}{|M(t)|}\right)^{\nu/2},
\end{align*}
which tends to $0$, $\mathbb{P}_0$-$a.s.$, since $\sum_{k\in M(t)}(\beta+\alpha X_{k-1}^2)/|M(t)|$ is convergent to a strictly positive constant and $\nu\in \{5,6,\cdots\}$.
This ensures the Lindeberg condition of Theorem \ref{theorem very base MCLT}.
Next, to ensure the second condition %\eqref{quaratic condition} 
of Theorem \ref{theorem empirical MCLT}, observe that
\[
 \frac{1}{|M(t)|}\sum_{k\in M(t)}\left(X_k^2-\mathbb{E}_{k-1}\left[X_k^2\right]\right)=\frac{1}{|M(t)|}\sum_{k\in M(t)}X_k^2-\beta -\alpha \frac{1}{|M(t)|}\sum_{k\in M(t)}X_{k-1}^2\to 0,
\]
$\mathbb{P}_0$-$a.s.$, due to $|M(t)|^{-1}\sum_{k\in M(t)}X_{k-1}^2\to \beta/(1-\alpha)$.
% Perhaps, using the explicit form
% \[
%  \mathbb{E}_{k-1}\left[|X_k-\mu|^4\right]=3\frac{n-2}{n-4}\left(\beta+\alpha X_{k-1}^2\right)^2 +6\mu^2\left(\beta+\alpha X_{k-1}^2\right)+\mu^4,
% \]
% ($\mu=0$, indeed)
% we need to show that $\limsup_{t\to+\infty}\sum_{k\in M(t)}(k-m(t-1))^{-2}X_{k-1}^p<+\infty$ for $p\in \{2,4\}$ 
% \textcolor{red}{Also, need to show 
% \[
%  v_1^2(t)-v^2_0(t)=\frac{\alpha}{|M(t)|}\sum_{k\in M(t)}(X_{k-1}^2-\mathbb{E}_0[X_{k-1}^2])=o(|M(t)|^{-q/(2+q)}).
% \]
% in $L^{1+q/2}(\Omega)$ for Theorem \ref{theorem implementable berry esseen}.
% }

\subsection{Adaptive control variates}\label{appendix CV}

The theoretical and conditional batch variances are given, respectively, by
\begin{align*}
 v^2_0(t)&=\frac{1}{|M(t)|}\sum_{k\in M(t)}\mathbb{E}_0\left[X_k^2\right]-\mu^2\\
 &=V(0)+2\left\langle \mathbb{E}_0[{\bm \theta}(t-1)], \mathbb{E}_0\left[\Psi(U_1)(U_1-\mathbbm{1}_d/2)\right]\right\rangle+\frac{1}{12}\mathbb{E}_0\left[\| {\bm \theta}(t-1)\|^2\right]-\mu^2,
\end{align*}
and
\begin{align*}
v_1^2(t)&=\frac{1}{|M(t)|} \sum_{k \in M(t)} \mathbb{E}_{k-1}\left[X_k^2\right]-\mu^2\\
&=V(0)+2\left\langle {\bm \theta}(t-1), \mathbb{E}_0\left[\Psi(U_1)(U_1-\mathbbm{1}_d/2)\right]\right\rangle+\frac{1}{12}\left\| {\bm \theta}(t-1)\right\|^2-\mu^2,
\end{align*}
for all $t\in\mathbb{N}$ and $k\in M(t)$.
By restricting the parameter domain to a sufficiently large yet compact one $\Theta(\subset \mathbb{R}^d)$ (which, at least, contains the initial state $0_d$ and the optimizer ${\bm \theta}^*$), it holds that %$\limsup_{t\to +\infty} v^2_0(t)<+\infty$.
%In addition, we thus get 
$v^2_0(t)\to V({\bm \theta}^*)-\mu^2$ as well as $v_1^2(t)\to V({\bm \theta}^*)-\mu^2$ by the bounded convergence theorem since the search domain $\Theta$ is compact with the aid of ${\bm \theta}(t)\to {\bm \theta}^*$, $\mathbb{P}_0$-$a.s.$
Hence, we get $\inf_{t\in\mathbb{N}}v^2_0(t)>0$ and $\sup_{t\in \mathbb{N}}v^2_0(t)<+\infty$ (Assumption \ref{standing assumption 2}), as well as $\sum_{k\in M(t)}(k-m(t-1))^{-2}{\rm Var}_0(X_k)\sim \mathbb{E}_0[V({\bm \theta}^*)]/|M(t)|\to 0$ (Theorem \ref{theorem law of large numbers}).
%Next, by further assuming $V({\bm \theta}^*)>\mu^2$, that is, perfect control variates is impossible due to \eqref{minimal second moment}, one can also ensure $\inf_{t\in\mathbb{N}}v^2_0(t)>0$.
One needs to set the batch sizes $\{M(t)\}_{t\in \mathbb{N}}$ satisfying $|M(t)|\to +\infty$ (Assumption \ref{standing assumption M}) and $|M(t+1)|\sim |M(t)|$, to ensure $v^2(t+1)/|M(t+1)|\sim v^2_0(t)/|M(t)|$ (Assumption \ref{standing assumption M}).
Clearly, the condition $v_1^2(t)/v^2_0(t)\stackrel{\mathbb{P}_0}{\to} 1$ (Theorem \ref{theorem base MCLT}) is satisfied.

% \textcolor{red}{
% \[
%  v_1^2(t)-v^2_0(t)=2\left\langle {\bm \theta}(t-1)-\mathbb{E}_0[{\bm \theta}(t-1)], \mathbb{E}_0\left[\Psi(U_1)(U_1-\mathbbm{1}_d/2)\right]\right\rangle+\frac{1}{12}\left(\| {\bm \theta}(t-1)\|^2-\mathbb{E}_0\left[\| {\bm \theta}(t-1)\|^2\right]\right)=o(|M(t)|^{-q/(2+q)},
% \]
% in $L^{1+q/2}(\Omega)$?
% Want to show that $|M(t)|^{q/2}\mathbb{E}_0[\|{\bm \theta}(t)-\mathbb{E}_0[{\bm \theta}(t)]\|^{1+q/2}]\to 0$ and $|M(t)|^{q/2}\mathbb{E}_0[|\|{\bm \theta}(t)\|^2-\mathbb{E}_0[\|{\bm \theta}(t)\|^2]|^{1+q/2}]\to 0$.}

Due to the quadratic integrability condition $\int_{(0,1)^d}|\Psi({\bf u})|^4d{\bf u}<+\infty$, it immediately holds that  $\limsup_{k\to +\infty}\mathbb{E}_0[|X_k|^4]<+\infty$ (Theorem \ref{theorem empirical MCLT}) and that $\mathbb{E}_{k-1}[|X_k-\mu|^4]$ is uniformly bounded from both below and above, since the search domain $\Theta$ and the uniform random vectors $\{U_k\}_{k\in \mathbb{N}}$ are bounded.
Hence, the Lindeberg condition (Theorem \ref{theorem very base MCLT}) holds by the convergence $|M(t)|^{-1}\sum_{k\in M(t)}\mathbb{E}_{k-1}[|X_k-\mu|^4]^{1/2}(\mathbb{P}_{k-1}(|X_k-\mu|>\epsilon \sqrt{|M(t)|}))^{1/2}\to 0$ (Theorem \ref{theorem very base MCLT}).
Finally, it is straightforward to show that $|M(t)|^{-1}\sum_{k\in M(t)}(X_k^2-\mathbb{E}_{k-1}[X_k^2])\to 0$, $\mathbb{P}_0$-$a.s.$ (Theorem \ref{theorem empirical MCLT}), due to ${\bm \theta}(t-1)\in \mathcal{F}_{k-1}$ and $U_k\perp\mathcal{F}_{k-1}$ for all $t\in \mathbb{N}$ and $k\in M(t)$.
%$\mathbb{E}_{k-1}[\Psi(U_k)(U_k-\mathbbm{I}_d/2)]=\mathbb{E}_0[\Psi(U_k)(U_k-\mathbbm{I}_d/2)]=\mathbb{E}_0[\Psi(U_1)(U_1-\mathbbm{I}_d/2)]$ and $\mathbb{E}_{k-1}[(U_k-\mathbbm{I}_d/2)^{\otimes 2}]=\mathbb{E}_0[(U_1-\mathbbm{I}_d/2)^{\otimes 2}]$

%Hence, the martingale central limit theorem \eqref{second MCLT} holds with the empirical variance $v_2^2(t)$ in \eqref{algorithm}, due to Theorem \ref{theorem empirical MCLT}.

\section{Comparisons with relevant stopping rules}\label{section comparisons with relevant stopping rules}

We here contrast the proposed framework with two relevant stopping rules in order to put the proposed ones into perspective.
In what follows, we denote by $\mathbb{D}((0,+\infty);\mathbb{R})$ the space of right-continuous functions in $\mathbb{R}$ with left limits defined on the open interval $(0,+\infty)$.

\subsection{Stopping rules in asymptotic regimes}

We have justified our formulations in an asymptotic manner along the way (Propositions \ref{proposition asymptotic validity of tau0} and \ref{proposition asymptotic validity 1})
%\ref{proposition asymptotic validity 2}
%and \ref{proposition asymptotic validity 3}) 
in a similar spirit to the existing framework in the asymptotic regime \cite{PW1992}.
A key assumption can be rephrased in our context that the batch variances are convergent.
That is, given that a limiting unit-time variance exits ($\lim_{t\to +\infty}v^2_0(t)(=:c)$), the strong consistency holds ($v_1^2(t)\stackrel{a.s.}{\to} c$, or $v_2^2(t)\stackrel{a.s.}{\to} c$), or the functional weak law of large numbers holds ($\{s^2(\lfloor t/\epsilon\rfloor):\,t\in (0,+\infty)\}\stackrel{\mathcal{L}}{\to} c$, or $\{\sigma^2(\lfloor t/\epsilon\rfloor):\,t\in (0,+\infty)\}\stackrel{\mathcal{L}}{\to} c$ in $\mathbb{D}((0,+\infty);\mathbb{R})$).
Clearly, all those conditions are stronger than ours (Assumption \ref{standing assumption 2} and the additional conditions in Proposition \ref{proposition asymptotic validity 1}). % and \ref{proposition asymptotic validity 2}).
%Hence, Propositions \ref{proposition asymptotic validity 1} and \ref{proposition asymptotic validity 2} hold true under those conditions.
Another key condition is the existence of a limiting continuous process of the scaled-centered estimation process weakly in $\mathbb{D}((0,+\infty);\mathbb{R}))$ in conjunction with the limiting unit-time variance.
To verify this condition in our context, consider the (non-batched) empirical mean $\mu_n:=n^{-1}\sum_{k=1}^n X_k$, instead of the batched one \eqref{empirical batch mean}, so that $\mu_n\to \mu$ almost surely by Theorem \ref{theorem law of large numbers}.
%In accordance with \cite{PW1992}, 
For $\epsilon\in (0,+\infty)$, define stochastic process $\{\mathcal{Y}_{\epsilon}(t):\,t\in (0,+\infty)\}$ by 
\[
\mathcal{Y}_{\epsilon}(t):=\epsilon^{-1/2}(\mu_{\lfloor t/\epsilon\rfloor}-\mu)=\epsilon^{-1/2}\frac{1}{\lfloor t/\epsilon\rfloor}\sum_{k=1}^{\lfloor t/\epsilon \rfloor}(X_k-\mathbb{E}_{k-1}[X_k]),\quad t\in (0,+\infty),
\]
which converges weakly to $\{W_t/t:\,t\in (0,+\infty)\}$ in $\mathbb{D}((0,+\infty);\mathbb{R})$, where $\{W_t:\,t\ge 0\}$ is a centered Brownian motion with ${\rm Var}_0(W_1)=c(=\lim_{t\to +\infty} v^2_0(t))$.
To derive this weak convergence, consider $\mathcal{X}_{\epsilon}(t):=\epsilon \lfloor t/\epsilon \rfloor \mathcal{Y}_{\epsilon}(t)=\epsilon^{1/2}\sum_{k=1}^{\lfloor t/\epsilon \rfloor}(X_k-\mathbb{E}_{k-1}[X_k])$.
The stochastic process $\{\mathcal{X}_{\epsilon}(t):\,t\in (0,+\infty)\}$ is centered with uncorrelated and asymptotically second-order stationary increments, due to  $\mathbb{E}_0[\mathcal{X}_{\epsilon}(t_1)]=0$, ${\rm Cov}_0(\mathcal{X}_{\epsilon}(t_2)-\mathcal{X}_{\epsilon}(t_1),\mathcal{X}_{\epsilon}(t_4)-\mathcal{X}_{\epsilon}(t_3))=0$ and 
%\quad {\rm Var}_0(\mathcal{X}_{\epsilon}(t_2)-\mathcal{X}_{\epsilon}(t_1))=\epsilon\sum_{k=\lfloor t_1/\epsilon\rfloor+1}^{\lfloor t_2/\epsilon\rfloor}{\rm Var}_0(X_k),
%\]
%Since $\lim_{\epsilon\to 0+}\epsilon (\lfloor t_2/\epsilon\rfloor -\lfloor t_1/\epsilon\rfloor)= t_2-t_1$, we have
%\[
%(t_2-t_1)\liminf_{k\to +\infty}{\rm Var}_0(X_k)\lesssim {\rm Var}_0(\mathcal{X}_{\epsilon}(t_2)-\mathcal{X}_{\epsilon}(t_1))\lesssim (t_2-t_1)\limsup_{k\to +\infty}{\rm Var}_0(X_k),
%\]
%where the asymptotic inequalities hold as $\epsilon\to 0+$.
%If, moreover, $\lim_{t\to +\infty}v^2_0(t)$ exists (and then is strictly positive due to Assumption \ref{standing assumption 2}), then 
\begin{align*} %begin{equation}\label{variance section 4.4}
 {\rm Var}_0(\mathcal{X}_{\epsilon}(t_2)-\mathcal{X}_{\epsilon}(t_1))
 &=\epsilon(\lfloor t_2/\epsilon\rfloor -\lfloor t_1/\epsilon\rfloor)\frac{1}{\lfloor t_2/\epsilon\rfloor -\lfloor t_1/\epsilon\rfloor}\sum_{k=\lfloor t_1/\epsilon\rfloor+1}^{\lfloor t_2/\epsilon\rfloor}{\rm Var}_0(X_k)\\
 &\to (t_2-t_1)c,
\end{align*}
for all $0\le t_1\le  t_2\le t_3\le t_4$.
%due to $\mathbb{E}_0[X_k]=\mathbb{E}_0[\mathbb{E}_{k-1}[X_k]]=\mu$.
To further show that the limiting process is a Brownian motion, let $\{\epsilon_n\}_{n\in\mathbb{N}}$ and $\{h_n\}_{n\in\mathbb{N}}$ be sequences of positive constants with $\epsilon_n\to 0$ and $h_n\to 0$ as $n\to +\infty$ and let $\tau_n$ be a stopping time with respect to the stochastic process $\{\mathcal{X}_{\epsilon_n}(t):\,t\ge 0\}$ for all $n\in \mathbb{N}$, that is, $\{\tau_n\le t\}\in \sigma(\{\mathcal{X}_{\epsilon_n}(s):\,s\in [0,t]\})$.
%\[
% \mathbb{P}_0(\lfloor \tau_n/\epsilon_n\rfloor\le m)\sim \mathbb{P}_0(\tau_n\le \epsilon_n m),
%\]
% With $\mathcal{K}_n:=\{\lfloor \tau_n/\epsilon_n\rfloor +1,\cdots,\lfloor (\tau_n +h_n)/\epsilon_n\rfloor\}$,
% \begin{align*}
% &\mathbb{P}_0\left(\left| \mathcal{X}_{\epsilon_n}(\tau_n+h_n)-\mathcal{X}_{\epsilon_n}(\tau_n)\right|>c\right)\\
% &\qquad \le c^{-2}\mathbb{E}_0\left[\left| \mathcal{X}_{\epsilon_n}(\tau_n+h_n)-\mathcal{X}_{\epsilon_n}(\tau_n)\right|^2\right]\\
% %&\qquad =c^{-2}\mathbb{E}_0\left[\mathcal{X}^2_{\epsilon_n}(\tau_n+h_n)-2\mathcal{X}_{\epsilon_n}(\tau_n+h_n)\mathcal{X}_{\epsilon_n}(\tau_n)+\mathcal{X}^2_{\epsilon_n}(\tau_n)\right]\\
% &\qquad = c^{-2}\epsilon_n\mathbb{E}_0\left[\left|\sum_{k\in \mathcal{K}_n}(X_k-\mathbb{E}_{k-1}[X_k])\right|^2\right]\\
% &\qquad =c^{-2}\epsilon_n  \mathbb{E}_0\left[\sum_{k\in \mathcal{K}_n}\mathbb{E}_{\lfloor \tau_n/\epsilon_n\rfloor}\left[(X_k-\mathbb{E}_{k-1}[X_k])^2\right]\right]+c^{-2}\epsilon_n  \mathbb{E}_0\left[\sum_{k_1\ne k_2 (\in \mathcal{K}_n)}\mathbb{E}_{\lfloor \tau_n/\epsilon_n\rfloor}\left[(X_{k_1}-\mathbb{E}_{k_1-1}[X_{k_1}])(X_{k_2}-\mathbb{E}_{k_2-1}[X_{k_2}])\right]\right]\\
% &\qquad =c^{-2}\epsilon_n  \mathbb{E}_0\left[\sum_{k\in \mathcal{K}_n}\mathbb{E}_{\lfloor \tau_n/\epsilon_n\rfloor}\left[(X_k-\mathbb{E}_{k-1}[X_k])^2\right]\right]\\
% %&\qquad \sim c^{-2}h_n  
% \end{align*}
Due to the martingale central limit theorems (Theorems \ref{theorem very base MCLT}, \ref{theorem base MCLT} and \ref{theorem empirical MCLT}), the limiting process is Gaussian and $\{\mathcal{X}_{\epsilon_n}(t):\,t\ge 0\}$ converges weakly in $\mathbb{D}([0,+\infty);\mathbb{R})$ to a centered Brownian motion, due to tightness $\mathbb{P}_0(| \mathcal{X}_{\epsilon_n}(\tau_n+h_n)-\mathcal{X}_{\epsilon_n}(\tau_n)|>m)\lesssim m^{-2}h_n \lim_{t\to +\infty}v^2_0(t)\to 0$ as $n\to +\infty$ for all $m>0$.
(See, for instance, \cite[Chapter 23]{kallenberg}.)

\subsection{Relative-precision stopping rules}

We here describe why a relative precision cannot be constructed in the present framework.
In the context of relative-precision stopping rules, the failure is represented by the probability $\mathbb{P}_{m(t-1)}(|\mu(t)-\mu|>\epsilon|\mu|)$ with $\epsilon\in (0,1)$, that is, the distance of the empirical mean from the true mean is larger than an $\epsilon$th fraction of the norm of the unknown mean $\mu$.
%In absolute-precision stopping rules \eqref{stopping based on general v(t)}, \eqref{stopping based on sigma(t)} and \eqref{empirical criterion based on berry-esseen}, on the one hand, the distance $|\mu(t)-\mu|$ does not appear per se, since its law is approximated by a normal distribution, respectively, based on the martingale central limit theorems \eqref{first MCLT}, \eqref{second MCLT} and \eqref{asymptotic Berry Esseen bound}.
%That is, the unknown true mean $\mu$ in the distance $|\mu(t)-\mu|$ is not needed and thus causes no implementation issues.
%On the other hand, in light of absolute-precision stopping rules \eqref{stopping based on general v(t)}, \eqref{stopping based on sigma(t)} and \eqref{empirical criterion based on berry-esseen}, it is clear that the (non-implementable) threshold $\epsilon|\mu|$ in the failure probability $\mathbb{P}_{m(t-1)}(|\mu(t)-\mu|>\epsilon|\mu|)$ appear as it is, in the corresponding relative-precision stopping rules.
%This issue needs to be addressed appropriately in constructing relative-precision stopping rules.
In a similar yet different manner to the progression \eqref{approximation by v_0(t)}, relative-precision stopping rules are to be designed under the assumption that the stopping batch $\tau$ is sufficiently large, as follows:
\begin{align*}
\delta<\mathbb{P}_{m(\tau)}\left(|\mu_{\star}(\tau)-\mu|>\epsilon|\mu|\right)
 &=\mathbb{P}_{m(\tau)}\left(\sqrt{|M(\tau)|}
 \frac{|\mu_{\star}(\tau)-\mu|}{v(\tau)}>\sqrt{|M(\tau)|}\frac{\epsilon|\mu|}{v(\tau)}\right)\\
% &\sim \mathbb{P}_{m(\tau)}\left(\sqrt{|M(\tau)|} \frac{|\mu_{\star}(\tau)-\mu|}{v(\tau)}>\sqrt{|M(\tau)|}\frac{\epsilon|\mu_{\star}(\tau)|}{v(\tau)}\right)\label{relative-precision derivation 1}\\
 &\sim 2\left(1-\Phi(\epsilon \sqrt{|M(\tau)|}|\mu|/v(\tau))\right),
\end{align*}
% \begin{align*}
%  \mathbb{P}_{m(\tau)}\left(|\mu_{\star}(\tau)-\mu|>\epsilon|\mu|\right)
%  &=\mathbb{P}_{m(\tau)}\left(\sqrt{|M(\tau)|}
%  \frac{|\mu_{\star}(\tau)-\mu|}{v(\tau)}>\sqrt{|M(\tau)|}\frac{\epsilon|\mu|}{v(\tau)}\right)\\
%  &\sim 2\left(1-\Phi(\epsilon \sqrt{|M(\tau)|}|\mu|/v(\tau))\right)\\
%  &\sim 2\left(1-\Phi(\epsilon \sqrt{|M(\tau)|}|\mu_{\star}(\tau)|/v(\tau))\right),
% \end{align*}
where we have employed the key fact that the regenerated batch $\{X_k^{\star}\}_{k\in M(\tau)}$ is conditionally independent of the stopping batch $\tau$ on the stopping-time $\sigma$-field $\mathcal{F}_{m(\tau)}$ for 
%Clearly, the first and second asymptotic equivalences are, respectively, due to the strong law of large numbers (Theorem \ref{theorem law of large numbers} (ii)) and 
the martingale central limit theorem (Theorem \ref{theorem very base MCLT}).
%, based upon the fact that the regenerated batch $\{X_k^{\star}\}_{k\in M(\tau)}$ is conditionally independent of the stopping batch $\tau$ on the stopping-time $\sigma$-field $\mathcal{F}_{m(\tau)}$.
The crucial issue here lies in how to treat the unknown mean $\mu$ in the rightmost term, which is exactly what one is looking for.
On the one hand, it cannot be replaced with the original empirical mean $\mu(\tau)$ because it is not clear whether $\mu(\tau)$ approximates $\mu$.
We stress, on the other hand, that $\mu$ there cannot be replaced with the regenerated empirical mean $\mu_{\star}(\tau)$ either despite it approximates $\mu$ on the stopping-time $\sigma$-field $\mathcal{F}_{m(\tau)}$, since then the stopping batch $\tau$ would be determined after generating the regenerated one $\{X_k^{\star}\}_{k\in M(\tau)}$ in reverse order.

Before closing this section, let us note, as already outlined along Assumption \ref{standing assumption 1}, that a relative-precision stopping rule was developed in \cite{LWP2013} upon a similar assumption to \eqref{base assumption}, but along with a boundedness condition on the underlying randomness (specifically, $X_n\in [0,1]$, $a.s.$), which plays an essential role for deriving bounds for the error probability upon which the stopping rule is built.

\end{document}